\documentclass[11pt]{article}
\usepackage[T1]{fontenc}
\usepackage{lmodern}
\usepackage[a4paper,margin=25mm]{geometry}
\usepackage{amsmath,amssymb,amsthm,mathtools}
\usepackage{microtype,booktabs,tabularx,array,enumitem}
\usepackage{tikz}
\usetikzlibrary{arrows.meta,calc,positioning}
\usepackage{flafter}
\usepackage[numbers,sort&compress]{natbib}
\usepackage[hidelinks]{hyperref}
\usepackage{url}

\newcommand{\cO}{\mathcal{O}}

\DeclareMathOperator*{\argmin}{arg\,min}
\DeclareMathOperator*{\argmax}{arg\,max}

\DeclarePairedDelimiter{\norm}{\lVert}{\rVert}

\DeclareFontFamily{U}{matha}{\hyphenchar\font45}
\DeclareFontShape{U}{matha}{m}{n}{
      <5> <6> <7> <8> <9> <10> gen * matha
      <10.95> matha10 <12> <14.4> <17.28> <20.74> <24.88> matha12
      }{}
\DeclareSymbolFont{matha}{U}{matha}{m}{n}
\DeclareFontSubstitution{U}{matha}{m}{n}

\DeclareFontFamily{U}{mathx}{\hyphenchar\font45}
\DeclareFontShape{U}{mathx}{m}{n}{
      <5> <6> <7> <8> <9> <10>
      <10.95> <12> <14.4> <17.28> <20.74> <24.88>
      mathx10
      }{}
\DeclareSymbolFont{mathx}{U}{mathx}{m}{n}
\DeclareFontSubstitution{U}{mathx}{m}{n}

\DeclareMathDelimiter{\vvvert}{0}{matha}{"7E}{mathx}{"17}

\numberwithin{equation}{section}
\newtheorem{theorem}{Theorem}[section]
\newtheorem{corollary}[theorem]{Corollary}
\newtheorem{lemma}[theorem]{Lemma}
\newtheorem{proposition}[theorem]{Proposition}

\theoremstyle{remark}
\newtheorem{remark}[theorem]{Remark}
\newcommand{\ResearchAgentSystem}{Colombo, our laboratory's internal auto-research system}
\newcommand{\R}{\mathbb R}
\newcommand{\eps}{\varepsilon}
\newcommand{\ip}[2]{\left\langle #1,#2\right\rangle}
\newcommand{\tO}{\widetilde{\mathcal O}}

\newcommand{\kap}{\kappa}
\newcommand{\calE}{\mathcal E}

\newcommand{\Mdet}{\mathfrak M_2^{\mathrm{det}}}
\newcommand{\Mrand}{\mathfrak M_2^{\mathrm{rand}}}

\usepackage{colortbl}
\newcommand{\authornames}{Fangqing Zhang, Haihan Zhang, Wendao Wu, Chenheng Zhang, Chunyuan Zheng, Cong Fang, Haoxuan Li, Zhouchen Lin}
\newcommand{\affiliation}{Peking University}
\newcommand{\email}{colombo.2026@outlook.com}
\newcommand{\mail}[1]{\href{mailto:#1}{\texttt{#1}}}
\newif\ifanonymous
\anonymousfalse
\title{\textbf{Matching Full-Hessian Oracle Complexity for\\
Nonconvex--Strongly-Concave Minimax Optimization}}
\ifanonymous
\author{Anonymous Authors}
\hypersetup{pdfauthor={Anonymous Authors}}
\else
\author{%
\parbox{\textwidth}{%
\centering
Fangqing Zhang$^{1,*}$\quad
Haihan Zhang$^{1,*}$\quad
Wendao Wu$^{1,*}$\\[0.28em]
Chenheng Zhang$^{1,*}$\quad
Chunyuan Zheng$^{1}$\\[0.28em]
Cong Fang$^{1,\dagger}$\quad
Haoxuan Li$^{1,\dagger}$\quad
Zhouchen Lin$^{1,\dagger}$\\[0.62em]
{\small $^{1}$\affiliation}\\[0.42em]
{\footnotesize
\mail{\email}
}\\[0.43em]
{\small $^{*}$Equal contribution.\qquad
$^{\dagger}$Corresponding authors.}
}%
}
\hypersetup{pdfauthor={\authornames}}
\fi
\hypersetup{pdftitle={Matching Full-Hessian Oracle Complexity for Nonconvex--Strongly-Concave Minimax Optimization}}
\date{}
\setlist[itemize]{leftmargin=*,itemsep=0.2em,topsep=0.2em}
\setlist[enumerate]{leftmargin=*,itemsep=0.2em,topsep=0.2em}
\begin{document}
\maketitle

\begin{abstract}
We characterize the oracle complexity of finding first-order
stationary points in smooth nonconvex--strongly-concave
(NC--SC) minimax optimization when each query provides
the full Hessian in addition to the function value and
gradient. For $\Phi(x)=\max_y f(x,y)$, assume that
$f$ has an $L$-Lipschitz gradient, a
$\rho$-Lipschitz Hessian, and is
$\mu$-strongly concave in $y$.
Let $\kappa=L/\mu$, let $\Delta$ denote the
initial value gap, and define
$r=\rho\kappa^{3/2}\varepsilon/L^2$.

We establish matching deterministic upper bounds
and randomized lower bounds, up to logarithmic
factors and additive initialization costs.
When the leading term dominates initialization,
the resulting minimax complexity is
\[
\widetilde{\Theta}\!\left(
\frac{\Delta L}{\varepsilon^2}\Psi_\kappa(r)
\right),
\qquad
\Psi_\kappa(r)=
\begin{cases}
r^{1/2}, & 0<r\le 1,\\
r^{2/7}, & 1<r<\kappa^{7/4},\\
\sqrt{\kappa}, & r\ge\kappa^{7/4}.
\end{cases}
\]
Despite the stronger oracle, the optimal
complexity exhibits a nontrivial dependence
on accuracy, conditioning, and Hessian smoothness.
In particular, these regimes exhibit accuracy
dependencies $\varepsilon^{-3/2}$,
$\varepsilon^{-12/7}$, and $\varepsilon^{-2}$,
respectively.

The upper bounds combine second-order regularity estimates
for a saddle Moreau envelope with certified accelerated
reduced-dual solves and amortized warm starts.
The lower bounds combine an all-derivative lift with a serial
composition of private-handle constructions, using hard instances
chosen independently of the algorithm's realized random seed.
The resulting characterization quantifies improvements over
the classical first-order benchmark under gradient smoothness
alone and identifies the regime in which that benchmark is recovered.
\end{abstract}

\par

{\small
\begin{samepage}
\noindent\textbf{AI Usage.}
Nearly the entire research pipeline for this paper was carried out by
\ResearchAgentSystem{}, powered by GPT-5.6 Sol. The system also conducted a
Lean-backed article audit of the resulting manuscript. The authors subsequently reviewed and approved the
mathematical claims, presentation, and formal artifacts, and take
responsibility for the final manuscript. The complete Lean audit report and the system's technical report will be made public at a later date.
\par
\end{samepage}
}
{\small
\begin{samepage}
\paragraph{Publication Intent and Recognition of Independent Work.}

This preprint is released as a research output of our study of
automated scientific research, for public scrutiny and discussion.
As we do not intend to submit this manuscript to a journal or
conference, it has not undergone the full editorial refinement
and internal review we would normally require for submission.
The authors remain responsible for its contents and for correcting
any errors or inaccurate attribution.

We recognize that releasing these results may affect researchers
who have independently been pursuing the same questions, including
work that is not yet public.
If this applies to your work, we invite you to contact us at
\email.
A brief description of the overlap, together with any dated
materials you are comfortable sharing, would help us understand
its scope and chronology.
Such materials might include draft excerpts, presentation slides,
or revision records; a complete manuscript or unpublished proof
is not required to begin the conversation.

We will review these cases carefully and update citations,
acknowledgments, and accounts of independent or concurrent work
as appropriate, including revising priority claims where warranted.
Nonpublic materials will be treated confidentially and used only
to assess the overlap unless their authors authorize otherwise.
We will obtain explicit permission from the authors before
publicly discussing such materials or providing them to an
automated research system.
\par
\end{samepage}
}

\section{Introduction}\label{sec:intro}

We study the unconstrained minimax problem
\[
  \min_{x\in\R^{d_x}}\Phi(x),
  \qquad
  \Phi(x):=\max_{y\in\R^{d_y}}f(x,y),
\]
when \(f\) can be nonconvex in \(x\) but is strongly concave in \(y\).
Strong concavity makes the inner maximizer unique and differentiable enough
for the value function to satisfy
\(\nabla\Phi(x)=\nabla_x f(x,y^\star(x))\).  The algorithmic target is a
true value-function stationary point,
\[
  \norm{\nabla\Phi(\widehat x)}\le\varepsilon,
\]
rather than a small joint gradient at an arbitrary primal--dual pair.

This problem is a standard model for adversarial learning, distributionally
robust optimization, and learning with an inner response variable.  Its
complexity also isolates a basic theoretical question: once a query returns
the full Hessian at the same unit cost as the gradient, how much can curvature
information improve the best possible dependence on accuracy, condition
number, and Hessian smoothness?  First-order work established the sharp
\(\widetilde\Theta(\Delta L\sqrt\kappa\,\varepsilon^{-2})\) frontier under
gradient smoothness alone~\cite{lin2020minimax,zhang2021complexity,li2021lower}.
That rate cannot describe the benefit of a Lipschitz Hessian or a bundled
second-order query.

The additional regularity does not lead to one uniformly optimal power law.
The dimensionless parameter
\[
  r=\frac{\rho\kappa^{3/2}\varepsilon}{L^2}
\]
compares the cubic scale induced by Hessian Lipschitzness with the
first-order scale of the value function.  We show that it divides the problem
into three regimes.  At small \(r\), an envelope cubic method attains the
nonconvex second-order rate \(\varepsilon^{-3/2}\), with the saddle geometry
contributing \(\kappa^{3/4}\).  At intermediate \(r\), the optimal rate is
\(\varepsilon^{-12/7}\); it comes from amortizing accelerated second-order
inner solves along an outer envelope-gradient trajectory.  At large \(r\),
second-order smoothness is too weak at the requested accuracy to improve on
the first-order NC--SC frontier.  Figure~\ref{fig:problem-region-map}
summarizes this three-regime structure.

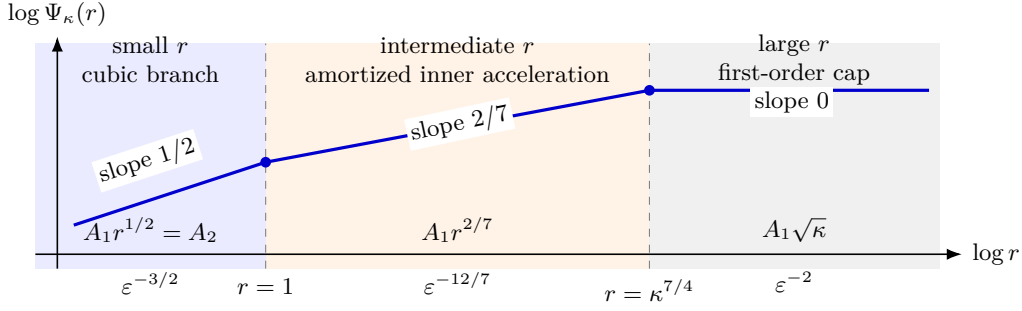
\begin{figure}[htbp]
\centering
\begin{tikzpicture}[
  x=1.45cm,
  y=0.95cm,
  >=Latex,
  every node/.style={font=\footnotesize},
  region/.style={align=center},
  rate/.style={fill=white,inner sep=1.5pt}
]
  \fill[blue!8] (-2.10,-1.48) rectangle (0,1.65);
  \fill[orange!10] (0,-1.48) rectangle (3.50,1.65);
  \fill[black!6] (3.50,-1.48) rectangle (6.15,1.65);

  \draw[dashed,gray] (0,-1.48) -- (0,1.65);
  \draw[dashed,gray] (3.50,-1.48) -- (3.50,1.65);
  \draw[->,semithick] (-2.10,-1.28) -- (6.35,-1.28)
    node[right] {\(\log r\)};
  \draw[->,semithick] (-1.90,-1.48) -- (-1.90,1.75)
    node[above] {\(\log \Psi_\kappa(r)\)};

  \draw[blue!80!black,very thick]
    (-1.75,-0.875) -- (0,0) -- (3.50,1.00) -- (6.05,1.00);
  \fill[blue!80!black] (0,0) circle (2pt);
  \fill[blue!80!black] (3.50,1.00) circle (2pt);

  \node[below] at (0,-1.48) {\(r=1\)};
  \node[below] at (3.50,-1.48) {\(r=\kappa^{7/4}\)};
  \node[region] at (-1.05,1.43) {small \(r\)\\cubic branch};
  \node[region] at (1.75,1.43) {intermediate \(r\)\\amortized inner acceleration};
  \node[region] at (4.82,1.43) {large \(r\)\\first-order cap};

  \node[rate,rotate=18.1] at (-1.08,0.02) {slope \(1/2\)};
  \node[rate,rotate=12] at (1.77,0.53) {slope \(2/7\)};
  \node[rate] at (4.80,0.84) {slope \(0\)};

  \node[region] at (-1.05,-0.94) {\(A_1r^{1/2}=A_2\)};
  \node[region] at (1.75,-0.94) {\(A_1r^{2/7}\)};
  \node[region] at (4.82,-0.94) {\(A_1\sqrt\kappa\)};
  \node[region,anchor=north] at (-1.05,-1.38) {\(\varepsilon^{-3/2}\)};
  \node[region,anchor=north] at (1.75,-1.38) {\(\varepsilon^{-12/7}\)};
  \node[region,anchor=north] at (4.82,-1.38) {\(\varepsilon^{-2}\)};
\end{tikzpicture}
\caption{Schematic problem-region map for the normalized complexity multiplier
\(\Psi_\kappa(r)\).  The horizontal and vertical coordinates are logarithmic;
only the two crossover locations and the three asymptotic slopes are
meaningful.  When \(\kappa=1\), the two crossover locations coincide and the
intermediate region collapses.}
\label{fig:problem-region-map}
\end{figure}

\paragraph{Contributions.}
The paper makes four contributions.
\begin{itemize}
\item We characterize the bundled full-Hessian minimax query complexity by a
single three-phase function \(A_1\Psi_\kappa(r)\), with matching deterministic
upper and randomized lower bounds up to logarithmic and initialization terms.
\item For the upper bound, we introduce a saddle Moreau envelope whose
Hessian is \(O(\rho\kappa^{3/2})\)-Lipschitz, but whose proximal saddle can be
solved through a reduced dual with \(O(\rho)\)-Lipschitz Hessian.  Eliminating
the primal block before accelerating the dual is essential: the eliminated
block has condition number at most three, and the reduced-dual contraction
cost is
\(\widetilde O(1+(\rho D/\mu)^{2/7})\), where $D$ is an upper bound on the initial distance
to the reduced-dual minimizer.
\item We amortize warm-start distances rather than paying their worst-case
value at every outer step.  The squared-step energy of envelope gradient
descent yields the middle branch, while the cubic-step energy yields the
high-accuracy branch.  A certified coupled A-HPE line search makes the
second-order inner solver valid under simultaneous gradient and Hessian
simulation errors.
\item For the lower bound, an all-derivative lift gives the small-\(r\)
branch.  For \(r\ge1\), a private-handle packet hides an exactly-flat tail
even after the handle becomes public.  A low-gradient endpoint difference
allows smooth relay gates to serialize \(\Theta(A_1)\) packets.  The product
argument is carried out on one fixed instance sampled independently of the
algorithm's random seed, which gives a per-instance randomized lower bound.
\end{itemize}

The full technical proof appears in the appendices.  The main text keeps the
problem specification, theorem, and the mechanisms that determine the three
rates.

\begin{table}[t]
\centering
\small
\setlength{\tabcolsep}{3.5pt}
\renewcommand{\arraystretch}{1.27}
\renewcommand{\tabularxcolumn}[1]{m{#1}}

\begin{tabularx}{\textwidth}{@{}
>{\centering\arraybackslash}m{0.19\textwidth}
>{\centering\arraybackslash}m{0.17\textwidth}
>{\centering\arraybackslash}m{0.37\textwidth}
>{\centering\arraybackslash}X@{}}

\toprule
\textbf{Work}
& \textbf{Oracle and target}
& \textbf{Complexity}
& \textbf{Setting} \\
\midrule

Lin et al.~\cite{lin2020minimax}
& First order; FOSP
& $\widetilde O(
\Delta L\sqrt\kappa\,\varepsilon^{-2})$
& NC--SC; \newline gradient smoothness \\

\addlinespace[2pt]

Zhang et al.~\cite{zhang2021complexity};\newline
Li et al.~\cite{li2021lower}
& First order; FOSP
& $\Omega(
\Delta L\sqrt\kappa\,\varepsilon^{-2})$
& General smoothness; \newline lower bounds \\

\addlinespace[2pt]

Wang et al.~\cite{wang2024higher}
& First order; FOSP
& $\widetilde O(
\sqrt\kappa L^{1/2}L_2^{1/4}
\varepsilon^{-7/4})$ (UB);\newline
$\Omega(
\sqrt\kappa L^{3/7}L_2^{2/7}
\varepsilon^{-12/7})$ (LB)
& Hessian-Lipschitz $\Phi$ \\

\addlinespace[2pt]

Luo et al.~\cite{luo2022mcn}
& Second order; SOSP
& $O(
\kappa^{3/2}\sqrt\rho\,
\varepsilon^{-3/2})$
& NC--SC; upper bound \\

\addlinespace[2pt]

Recent SOSP methods~
\cite{yao2024trustregion,
wang2025gradientnorm,
chen2026hsda,
ma2026singleloop}
& Second order or HVP; SOSP
& $\widetilde O(
\varepsilon^{-3/2})$
iterations
& Method-dependent factors \\

\addlinespace[2pt]

Carmon et al.~\cite{carmon}
& All derivatives; FOSP
& $\Omega\!\left(
\Delta\min\{
L\varepsilon^{-2},
\sqrt\rho\,\varepsilon^{-3/2}
\}\right)$
& Single-level nonconvex \\

\addlinespace[2pt]
\midrule

\textbf{This work}
& Bundled second order; FOSP
& $\widetilde\Theta(
A_1\Psi_\kappa(r))$
& Deterministic UB; \newline randomized LB \\

\bottomrule
\end{tabularx}

\caption{
Comparison of representative oracle-complexity results for NC--SC minimax optimization and related single-level nonconvex optimization problem. Here, \(L_2\) denotes the Hessian Lipschitz constant of the value function \(\Phi\), FOSP and SOSP denote first- and
second-order stationary points, respectively;
UB and LB denote upper and lower bounds.
Rates follow the respective oracle models
and smoothness assumptions.
}
\label{tab:rate-comparison}
\end{table}
\section{Related Work}\label{sec:related}
Table~\ref{tab:rate-comparison} summarizes the most
relevant oracle-complexity results and highlights the
distinction between the present work and existing
first- and second-order methods.

\paragraph{First-order NC--SC complexity.}
Lin, Jin, and Jordan~\cite{lin2020minimax} used accelerated proximal-point
ideas to obtain a near-optimal
\(\widetilde O(\Delta L\sqrt\kappa\,\varepsilon^{-2})\) gradient bound for
value-function stationarity.  Zhang et al.~\cite{zhang2021complexity}
established the corresponding \(\Omega(\Delta L\sqrt\kappa\,\varepsilon^{-2})\)
lower bound in the general smooth setting and gave an acceleration scheme
matching it up to logarithms.  Li et al.~\cite{li2021lower} independently
proved a \(\sqrt\kappa\,\varepsilon^{-2}\) deterministic first-order lower
bound under their normalization.  These results identify the large-\(r\)
endpoint of our frontier, but they neither charge nor exploit a full-Hessian
query.

\paragraph{Higher smoothness and second-order information.}
Wang, Zhang, and Zhang~\cite{wang2024higher} showed that a first-order method
can exploit a Lipschitz Hessian of \(\Phi\), obtaining an
\(\varepsilon^{-7/4}\) upper dependence and an
\(\varepsilon^{-12/7}\) lower dependence under the source's normalization.
Their oracle remains first order and their smoothness assumption is imposed
directly on the value function.  Luo, Li, and Chen~\cite{luo2022mcn} studied
a different target: an approximate second-order stationary point of
\(\Phi\).  Their Minimax Cubic Newton method uses second-order oracle calls
with \(O(\kappa^{3/2}\sqrt\rho\,\varepsilon^{-3/2})\) dependence.  Our target
is first-order stationarity, our Hessian-Lipschitz assumption is on the joint
function \(f\), and the goal is a matching minimax query characterization.
Subsequent trust-region, gradient-norm-regularized, homogeneous, and
single-loop Newton schemes retain an
\(\widetilde O(\varepsilon^{-3/2})\) iteration dependence for approximate
second-order stationarity, while addressing subproblem structure, Hessian-
vector implementation, or the inner-loop bottleneck in different
ways~\cite{yao2024trustregion,wang2025gradientnorm,chen2026hsda,ma2026singleloop}.
These are algorithmic upper bounds for the stronger SOSP target; none gives a
matching bundled-oracle lower bound for the FOSP criterion studied here.

\paragraph{Second-order acceleration and lower-bound tools.}
The upper proof uses the large-step accelerated hybrid proximal
extragradient framework of Monteiro and Svaiter~\cite{monteiro2013ahpe} and
the implementable high-order line-search machinery of Jiang, Wang, and
Zhang~\cite{jiang2021tensor}.  The high-accuracy lower branch starts from the
randomized all-derivative stationary-point lower bound of Carmon et
al.~\cite{carmon}.  Our new work lies in the NC--SC reductions, the
reduced-dual oracle simulation, the warm-start accounting, and the
private-handle serial construction.

\section{Problem Setup}\label{sec:setup}

Let \(f:\R^{d_x}\times\R^{d_y}\to\R\) be twice continuously
differentiable.  All norms are Euclidean and operator norms are induced by
the Euclidean norm.  For \(z=(x,y)\), assume
\[
 \norm{\nabla f(z)-\nabla f(z')}\le L\norm{z-z'},\qquad
 \norm{\nabla^2f(z)-\nabla^2f(z')}\le\rho\norm{z-z'}.
\]
For every \(x\), the function \(f(x,\cdot)\) is
\(\mu\)-strongly concave.  Hence
\[
 y^\star(x)=\argmax_y f(x,y),\qquad
 \Phi(x)=f(x,y^\star(x)),\qquad
 \kappa=\frac L\mu
\]
are well defined, with \(\kappa\ge1\).  The starting pair obeys
\[
 \Phi(x_0)-\inf_x\Phi(x)\le\Delta,
 \qquad
 \norm{y_0-y^\star(x_0)}\le R_y.
\]

\paragraph{Oracle and computation model.}
One charged query at any \((x,y)\) returns the exact bundle
\[
 \mathcal O_2(x,y)
 =\bigl(f(x,y),\nabla f(x,y),\nabla^2f(x,y)\bigr).
\]
All measurable computation using returned real numbers and matrices is free,
including exact linear algebra and exact minimization of models constructed
from the transcript.  Algorithms may choose arbitrary adaptive query points
and may return an unqueried point.  A randomized algorithm uses an
instance-independent random seed and must return \(\widehat x\) satisfying
\(\norm{\nabla\Phi(\widehat x)}\le\varepsilon\) with probability at least
\(2/3\) on every fixed legal instance.  The worst case ranges over every
finite dimension.  We write \(\Mdet(\varepsilon)\) and
\(\Mrand(\varepsilon)\) for the deterministic and randomized minimax query
complexities.

\paragraph{Dimensionless regimes and initialization.}
Define
\[
 A_1=\Delta L\varepsilon^{-2},\qquad
 A_2=\Delta\sqrt\rho\,\kappa^{3/4}\varepsilon^{-3/2},\qquad
 r=\frac{\rho\kappa^{3/2}\varepsilon}{L^2},\label{eq:def-r}
\]
so that \(A_2=A_1\sqrt r\), and let
\[
 \Psi_\kappa(r)=
 \begin{cases}
 r^{1/2},&0<r\le1,\\
 r^{2/7},&1\le r\le\kappa^{7/4},\\
 \kappa^{1/2},&r\ge\kappa^{7/4}.
 \end{cases}\label{eq:Psi}
\]
The one-time cost of bringing the initial dual point to the first proximal
saddle is controlled by
\[
 D_{y,0}=R_y+\kappa\sqrt{\Delta/L},\qquad
 I_0=\widetilde{\mathcal O}\!\left(
 1+\min\left\{\sqrt\kappa,
 \left(\frac{\rho D_{y,0}}\mu\right)^{2/7}\right\}\right).\label{eq:init}
\]

\section{Main Results}\label{sec:results}

\begin{theorem}[Three-phase frontier]\label{thm:main}
There are universal constants \(c,C>0\) with the following properties.
The deterministic bundled-oracle complexity satisfies
\[
 \Mdet(\varepsilon)
 \le \widetilde{\mathcal O}\!\left(
 I_0+\min\left\{
 A_2(1+r^{1/7}),\ A_1(1+r^{2/7}),\ \sqrt\kappa A_1
 \right\}\right).\label{eq:main-upper}
\]
If \(A_1\ge C\), then every randomized algorithm with a query budget below
the following lower bound fails with probability at least \(1/3\) on some
deterministic legal instance chosen independently of its realized random
seed.  The hard instance has exact dual initialization \(R_y=0\), and
\[
 \Mrand(\varepsilon)
 \ge cA_1\Psi_\kappa(r).\label{eq:main-lower}
\]
Consequently, whenever the main term dominates \(I_0+1\),
\[
 \Mrand(\varepsilon)
 \asymp_{\log}
 \Mdet(\varepsilon)
 \asymp_{\log} A_1\Psi_\kappa(r).\label{eq:main-frontier}
\]
Equivalently, the three rates are
\[
 \begin{cases}
 \Delta\sqrt\rho\,\kappa^{3/4}\varepsilon^{-3/2},&r\le1,\\[1mm]
 \Delta\rho^{2/7}\kappa^{3/7}L^{3/7}\varepsilon^{-12/7},
       &1\le r\le\kappa^{7/4},\\[1mm]
 \Delta L\sqrt\kappa\,\varepsilon^{-2},&r\ge\kappa^{7/4}.
 \end{cases}\label{eq:three-rates}
\]
\end{theorem}

The theorem is a statement about the exact bundled oracle.  In particular,
it does not count gradients and Hessians separately, does not assume a
finite-sum structure, and does not bound the dimensions.  The lower bound is
therefore allowed to choose a sufficiently high but finite dimension.

\paragraph{External interfaces.}
Three general results are used as named black boxes: the randomized
all-derivative stationary-point lower bound of Carmon et
al.~\cite{carmon}, the large-step A-HPE convergence theorem of Monteiro and
Svaiter~\cite{monteiro2013ahpe}, and the implementable coupled
\(\beta\)-line search of Jiang, Wang, and Zhang~\cite{jiang2021tensor}.
Every NC--SC-specific reduction, simulation, packet construction, shell,
and fixed-instance coupling is proved in the appendices.

\section{Proof Overview and Main Mechanisms}\label{sec:overview}

\subsection{Upper bound: envelope reduction and amortized inner solves}

For an outer center \(x\), define the saddle Moreau envelope
\[
 G_x(u,y)=f(u,y)+L\norm{u-x}^2,
 \qquad
 \calE(x)=\min_u\max_yG_x(u,y).
\]
If \((u_x,v_x)\) is its unique saddle, the saddle first-order conditions
and Danskin's theorem give the exact transfer
\[
 \nabla\calE(x)=2L(x-u_x)=\nabla\Phi(u_x).\label{eq:overview-transfer}
\]
Thus an envelope stationary point yields a value-function stationary point
without a surrogate-stationarity loss.  Implicit differentiation further
gives
\[
 \mathrm{Lip}(\nabla\calE)\le2L,
 \qquad
 \mathrm{Lip}(\nabla^2\calE)\le8\rho(1+\kappa)^{3/2}.
\]

The proximal saddle is not solved as a joint root problem.  For fixed
\((x,y)\), eliminate
\[
 u_x(y)=\argmin_u G_x(u,y),
 \qquad
 \psi_x(y)=-\min_uG_x(u,y).
\]
The primal Hessian lies between \(LI\) and \(3LI\), while the reduced dual
satisfies
\[
 \mu I\preceq\nabla^2\psi_x(y)\preceq2LI,
 \qquad
 \mathrm{Lip}(\nabla^2\psi_x)=O(\rho).
\]
The accelerated second-order contraction derived in
Appendix~\ref{sec:warm-solve} therefore solves a proximal saddle from dual
distance \(D\) in
\[
 \widetilde O\!\left(
 1+\min\left\{\sqrt\kappa,
 \left(\frac{\rho D}{\mu}\right)^{2/7}\right\}
 \right)
\]
bundled calls.  Let \(\widetilde v_{t-1}\) be the certified reduced-dual
approximation returned at center \(x_{t-1}\), and define
\(\tau_{t-1}=\norm{\widetilde v_{t-1}-v_{x_{t-1}}}\) and
\(D_t=\norm{\widetilde v_{t-1}-v_{x_t}}\).  When
\(\widetilde v_{t-1}\) initializes the reduced-dual solve at \(x_t\),
consecutive centers satisfy the warm-start estimate
\[
 D_t\le\tau_{t-1}+2\sqrt\kappa\norm{x_t-x_{t-1}}.
\]

There are three outer choices.  Envelope gradient descent uses
\(N_1=O(A_1)\) iterations and
\(\sum_t\norm{x_{t+1}-x_t}^2=O(\Delta/L)\).  H\"older's inequality applied
to the \(2/7\)-power inner cost gives
\[
 Q_{\mathrm{GD}}
 \le\widetilde O\bigl(A_1(1+r^{2/7})\bigr).
\]
Envelope cubic regularization uses \(N_2=O(A_2)\) iterations and
\(\sum_t\norm{x_{t+1}-x_t}^3=O(\Delta/(\rho\kappa^{3/2}))\), which gives
\[
 Q_{\mathrm{cubic}}
 \le\widetilde O\bigl(A_2(1+r^{1/7})\bigr).
\]
Finally, using accelerated first-order inner solves caps the total cost by
\(\widetilde O(\sqrt\kappa A_1)\).  Taking the minimum yields
\eqref{eq:main-upper}; the crossover points are \(r=1\) and
\(r=\kappa^{7/4}\).

\subsection{Lower bound: two constructions and one fixed instance}

For \(r\le1\), start from a smooth single-level function that is hard even
when every derivative is revealed.  The quadratic coupling
\[
 f(x,y)=F(y)-\mu\norm{y-bx-c}^2,
 \qquad
 b^2=\frac{\kappa-3}{2},
\]
preserves the joint \((L,\rho)\) promises and transfers small
\(\norm{\nabla\Phi(x)}\) to small \(\norm{\nabla F(bx+c)}\).  This yields
\[
 \Omega\!\left(
 \Delta\min\left\{L\varepsilon^{-2},
 \sqrt\rho\,\kappa^{3/4}\varepsilon^{-3/2}\right\}
 \right),
\]
which is the small-\(r\) branch and the baseline beyond it.

For \(r\ge1\), choose a packet length
\[
 q\asymp\min\{r^{2/7},\sqrt\kappa\}.
\]
An exactly-flat chain hides \(q\) orthogonal directions.  Its public
quadratic amplifier exposes a handle and the first path direction but cannot
generate the tail by repeated Hessian algebra.  A Fenchel lift turns distance
from the hidden packet minimizer into value-function gradient.  Crucially,
the active and inactive endpoints differ only through a bounded handle term
whose value, gradient, Hessian, and Hessian-Lipschitz scales are small enough
to pass through a smooth relay gate.

The relay serializes \(T\asymp A_1\) independent packets: true stationarity
forces every preceding gate to open, while one bundled response advances at
most one hidden tail direction in the currently active packet.  Adaptive
Haar concentration is applied to virtual transcripts that never condition
future disclosure on a no-hit event.  A conditional product argument then
selects one deterministic instance, independently of the algorithm's seed,
on which at least \(Tq\) queries are necessary with constant probability.
Hence
\[
 \Mrand(\varepsilon)
 \ge \Omega\!\left(A_1\min\{r^{2/7},\sqrt\kappa\}\right),
\]
which joins the all-derivative branch at \(r=1\) and saturates at
\(r=\kappa^{7/4}\).

\section{Discussion}\label{sec:discussion}

The theorem shows that Hessian access helps only after its smoothness is
measured relative to the target accuracy and the saddle condition number.
The intermediate exponent \(12/7\) is not an isolated property of an outer
method; it is the balance between an \(O(A_1)\) envelope trajectory and a
distance-sensitive \(2/7\)-power inner contraction.  The matching lower
packet reproduces the same balance through the relation
\(q^{7/2}\asymp r\).

Several boundaries remain.  The oracle is exact and bundles value, gradient,
and the full joint Hessian at one unit cost; separate gradient/Hessian costs,
Hessian-vector products, stochastic or finite-sum oracles, and finite-
dimension refinements require different characterizations.  The result hides
logarithmic factors and records a separate initialization term \(I_0\).
Finally, the proof treats the three general results listed after
Theorem~\ref{thm:main} as external interfaces.  All problem-specific
reductions and randomized quantifier arguments are included below.

\appendix
\section{Complete Technical Proofs}\label{sec:complete-proofs}

This appendix gives the complete upper- and lower-bound arguments.  It
retains the local certificates, line-search implementation, smooth shells,
virtual filtrations, and fixed-instance coupling that were summarized in
Section~\ref{sec:overview}.

\section{Saddle Moreau envelope and exact stationarity transfer}\label{sec:envelope}
For a center \(x\), define
\[
 G_x(u,y)=f(u,y)+L\norm{u-x}^2.\label{eq:G}
\]
Because \(\nabla^2_{uu}f\succeq-LI\), \(G_x(\cdot,y)\) is \(L\)-strongly convex; it remains \(\mu\)-strongly concave in \(y\).  Let
\[
 u_x=\argmin_u\max_yG_x(u,y),
 \qquad
 v_x=\argmax_yG_x(u_x,y),
 \qquad
 z_x=(u_x,v_x),
 \]
and define
\[
 \calE(x)=\min_u\max_yG_x(u,y).\label{eq:envelope}
\]
Equivalently,
\[
 \calE(x)=\min_u\{\Phi(u)+L\norm{u-x}^2\}.
\]

\begin{lemma}[Exact transfer and gap]\label{lem:transfer}
For every \(x\),
\[
 \nabla\calE(x)=2L(x-u_x)=\nabla\Phi(u_x),\label{eq:transfer}
\]
and
\[
 \calE(x_0)-\inf\calE\le\Delta,
 \qquad
 \inf\calE=\inf\Phi.\label{eq:E-gap}
\]
\end{lemma}

\begin{proof}
The saddle first-order conditions are
\[
 \nabla_u f(u_x,v_x)+2L(u_x-x)=0,
 \qquad
 \nabla_y f(u_x,v_x)=0.
\]
Thus \(v_x=y^\star(u_x)\).  The envelope theorem gives \(\nabla\calE(x)=2L(x-u_x)\), which equals \(\nabla\Phi(u_x)\).  The gap identities follow by choosing \(u=x_0\) for the upper comparison and \(x=u\) when taking infima.
\end{proof}

Let \(H_{uu},H_{uy},H_{yy}\) denote the Hessian blocks of \(f\) at \((u_x,v_x)\).  For a direction \(h\), write
\[
 (p_h,q_h)=Dz_x[h].
\]
Differentiating the saddle equations gives
\begin{align}
 (H_{uu}+2LI)p_h+H_{uy}q_h&=2Lh,\label{eq:tangent-u}\\
 H_{yu}p_h+H_{yy}q_h&=0.\label{eq:tangent-y}
\end{align}

\begin{lemma}[Square-root saddle sensitivity]\label{lem:sens}
For every \(h\),
\[
 \norm{p_h}\le2\norm h,
 \qquad
 \norm{q_h}\le2\sqrt\kap\norm h,
 \qquad
 \norm{Dz_x[h]}\le2\sqrt{1+\kap}\norm h.\label{eq:sens}
\]
Consequently,
\[
 \norm{v_x-v_{x'}}\le2\sqrt\kap\norm{x-x'}.\label{eq:v-sens}
\]
\end{lemma}

\begin{proof}
Put \(A=-H_{yy}\succeq\mu I\).  Equation~\eqref{eq:tangent-y} gives \(H_{yu}p_h=Aq_h\).  Pairing \eqref{eq:tangent-u} with \(p_h\) yields
\[
 p_h^\top(H_{uu}+2LI)p_h+q_h^\top A q_h
 =2L\ip h{p_h}.
\]
Since \(H_{uu}\succeq-LI\),
\[
 L\norm{p_h}^2+\mu\norm{q_h}^2
 \le2L\norm h\norm{p_h}.
\]
The displayed bounds follow by first dropping the \(q_h\) term and then substituting the resulting estimate for \(p_h\).  Integrating along a line segment gives \eqref{eq:v-sens}.
\end{proof}

Eliminating \(q_h\) gives
\[
 p_h=2LS_x^{-1}h,
 \qquad
 S_x=H_{uu}+2LI-H_{uy}H_{yy}^{-1}H_{yu}\succeq LI,\label{eq:Schur}
\]
and hence
\[
 \nabla^2\calE(x)=2LI-4L^2S_x^{-1}.\label{eq:E-hess}
\]

\begin{lemma}[Envelope regularity]\label{lem:E-reg}
The envelope has \(2L\)-Lipschitz gradient and an \(M_{\calE}\)-Lipschitz Hessian with
\[
 M_{\calE}\le8\rho(1+\kap)^{3/2}.\label{eq:ME}
\]
\end{lemma}

\begin{proof}
Since \(S_x\succeq LI\), \(0\preceq Du_x=2LS_x^{-1}\preceq2I\).  Formula~\eqref{eq:E-hess} gives \(\norm{\nabla^2\calE}\le2L\).  If \(f\) is \(C^3\), implicit differentiation gives the reduced third-derivative identity
\[
 D^3\calE(x)[h_1,h_2,h_3]
 =D^3f(z_x)[Dz_xh_1,Dz_xh_2,Dz_xh_3].\label{eq:third-id}
\]
Lemma~\ref{lem:sens} and \(\norm{D^3f}\le\rho\) therefore give \eqref{eq:ME}.  The \(C^{2,1}\) case follows by mollification or almost-everywhere differentiation and line integration.
\end{proof}

\section{Reduced-dual elimination and certified second-order solves}\label{sec:dual-reduction}
The key improvement over direct joint root tracking is that the proximal primal variable is uniformly well conditioned.  Fix both \(x\) and \(y\) and define
\[
 u_x(y)=\argmin_uG_x(u,y),
 \qquad
 d_x(y)=\min_uG_x(u,y)=G_x(u_x(y),y),
 \qquad
 \psi_x(y)=-d_x(y).\label{eq:dual-def}
\]
Then \(v_x=\argmin_y\psi_x(y)\) and \(u_x=u_x(v_x)\).

\begin{lemma}[Constant-conditioned primal elimination]\label{lem:primal-elim}
For every \((x,y)\), put
\[
 A_x(y)=\nabla^2_{uu}f(u_x(y),y)+2LI.
\]
Then
\[
 LI\preceq A_x(y)\preceq3LI,\label{eq:A-bounds}
\]
so \(u\mapsto G_x(u,y)\) has condition number at most \(3\).  Moreover,
\[
 D_yu_x(y)=-A_x(y)^{-1}\nabla^2_{uy}f(u_x(y),y),
 \qquad
 \norm{D_yu_x(y)}\le1.\label{eq:u-y-lip}
\]
Thus \(y\mapsto(u_x(y),y)\) is \(\sqrt2\)-Lipschitz.
\end{lemma}

\begin{proof}
The first assertion follows from \(-LI\preceq\nabla^2_{uu}f\preceq LI\).  Differentiate the equation
\[
 \nabla_u f(u_x(y),y)+2L(u_x(y)-x)=0
\]
with respect to \(y\) and use \(\norm{A_x(y)^{-1}}\le1/L\) and \(\norm{\nabla^2_{uy}f}\le L\).
\end{proof}

\begin{lemma}[Reduced-dual curvature]\label{lem:dual-curvature}
The reduced dual satisfies
\[
 \nabla\psi_x(y)=-\nabla_yf(u_x(y),y),\label{eq:dual-grad}
\]
\[
 \nabla^2\psi_x(y)
 =-H_{yy}+H_{yu}A_x(y)^{-1}H_{uy},\label{eq:dual-hess}
\]
where all blocks are evaluated at \((u_x(y),y)\).  Consequently,
\[
 \mu I\preceq\nabla^2\psi_x(y)\preceq2LI.\label{eq:dual-sc-smooth}
\]
There is a universal constant \(C_d\) such that
\[
 \norm{\nabla^2\psi_x(y)-\nabla^2\psi_x(y')}
 \le C_d\rho\norm{y-y'}.\label{eq:dual-Hlip}
\]
One may take \(C_d=4\sqrt2\).
\end{lemma}

\begin{proof}
The envelope theorem gives \eqref{eq:dual-grad}; differentiating it and using \eqref{eq:u-y-lip} gives \eqref{eq:dual-hess}.  Since \(-H_{yy}\succeq\mu I\) and the Schur term is positive semidefinite, the lower bound in \eqref{eq:dual-sc-smooth} follows.  For the upper bound,
\[
 \norm{-H_{yy}}\le L,
 \qquad
 \norm{H_{yu}A_x(y)^{-1}H_{uy}}\le L.
\]
For \eqref{eq:dual-Hlip}, let \(z=(u_x(y),y)\) and \(z'=(u_x(y'),y')\).  Lemma~\ref{lem:primal-elim} gives \(\norm{z-z'}\le\sqrt2\norm{y-y'}\).  Every Hessian block changes by at most \(\rho\norm{z-z'}\).  The resolvent identity and \(A_x(y),A_x(y')\succeq LI\) give
\[
 \norm{A_x(y)^{-1}-A_x(y')^{-1}}
 \le L^{-2}\rho\norm{z-z'}.
\]
Expanding the difference of the three-factor Schur terms yields three contributions, each at most \(\rho\norm{z-z'}\).  Together with the \(yy\) block this gives \(4\rho\norm{z-z'}\le4\sqrt2\rho\norm{y-y'}\).
\end{proof}

\begin{lemma}[Simulation of a reduced-dual second-order oracle]\label{lem:dual-oracle}
A value/gradient/Hessian query to \(\psi_x\) at \(y\) can be simulated to any prescribed polynomial accuracy using \(\tO(1)\) calls to \(\cO_2\).  More precisely, if \(\widetilde u\) satisfies \(\norm{\widetilde u-u_x(y)}\le e_u\), then the natural plug-in estimates obey
\[
 |\widetilde\psi_x(y)-\psi_x(y)|\le\frac{3L}{2}e_u^2,\label{eq:dual-value-error}
\]
\[
 \norm{\widetilde g_x(y)-\nabla\psi_x(y)}\le Le_u,
 \qquad
 \norm{\widetilde H_x(y)-\nabla^2\psi_x(y)}\le C\rho e_u\label{eq:dual-oracle-error}
\]
for a universal \(C\).  The logarithmic cost of producing \(e_u\) is hidden by \(\tO(1)\).
\end{lemma}

\begin{proof}
The primal subproblem is \(L\)-strongly convex and \(3L\)-smooth.  Its observable residual \(r_u=\nabla_uG_x(\widetilde u,y)\) certifies \(\norm{\widetilde u-u_x(y)}\le\norm{r_u}/L\), so accelerated gradient descent reaches a certified distance \(e_u\) in a logarithmic number of original oracle calls.  Warm starts are available because \(u_x(y)\) is \(1\)-Lipschitz in \(y\) and \(2\)-Lipschitz in \(x\).  The value estimate follows from \(3L\)-smoothness and stationarity at \(u_x(y)\).  The gradient estimate follows from \(L\)-smoothness of \(f\).  For the Hessian, use the Schur formula \eqref{eq:dual-hess}, Hessian Lipschitzness, the inverse bound \(\norm{A^{-1}}\le1/L\), and the resolvent identity exactly as in the proof of Lemma~\ref{lem:dual-curvature}.  Choosing \(e_u\) below any inverse-polynomial oracle tolerance required
by the certified line search costs only an additional logarithm.
\end{proof}

\begin{lemma}[Certified approximate gradients are enlarged subgradients]
\label{lem:gradient-enlargement}
Let \(g\) be differentiable and \(\mu\)-strongly convex.  If
\(\norm{v-\nabla g(y)}\le\delta\), then
\[
 v\in\partial_{\epsilon_h}g(y),
 \qquad
 \epsilon_h=\frac{\delta^2}{2\mu}.\label{eq:gradient-enlargement}
\]
\end{lemma}

\begin{proof}
For every \(w\), strong convexity gives
\[
 g(w)\ge g(y)+\ip{\nabla g(y)}{w-y}
       +\frac\mu2\norm{w-y}^2.
\]
Writing \(e=v-\nabla g(y)\) and minimizing
\(\mu\norm{w-y}^2/2-\ip e{w-y}\) over \(w-y\) yields
\[
 g(w)\ge g(y)+\ip v{w-y}-\frac{\norm{e}^2}{2\mu},
\]
which is exactly \eqref{eq:gradient-enlargement}.
\end{proof}

\begin{theorem}[Large-step A-HPE black box]
\label{thm:large-step-ahpe-blackbox}
Fix \(0<\sigma<1\) and \(\theta>0\).  Let \(g\) be closed convex, attain its
minimum, and have an \(M\)-Lipschitz Hessian.  The A-HPE recurrence starts
from \(A_0=0\) and \(x_0=y_0\).  Given \(A_k,x_k,y_k\), it chooses
\(\lambda_{k+1}>0\) and sets
\[
 a_{k+1}=\frac{\lambda_{k+1}+
 \sqrt{\lambda_{k+1}^2+4\lambda_{k+1}A_k}}2,
 \qquad A_{k+1}=A_k+a_{k+1},\label{eq:ahpe-a}
\]
\[
 z_k=\frac{A_k}{A_{k+1}}y_k+
       \frac{a_{k+1}}{A_{k+1}}x_k.\label{eq:ahpe-center}
\]
It then produces \(\widetilde y_{k+1},v_{k+1}\) and
\(\epsilon_{k+1}\ge0\) such that
\[
 v_{k+1}\in\partial_{\epsilon_{k+1}}g(\widetilde y_{k+1}),
\]
\[
 \norm{\lambda_{k+1}v_{k+1}+\widetilde y_{k+1}-z_k}^2
 +2\lambda_{k+1}\epsilon_{k+1}
 \le\sigma^2\norm{\widetilde y_{k+1}-z_k}^2,\label{eq:ahpe-relative}
\]
updates
\[
 x_{k+1}=x_k-a_{k+1}v_{k+1},
 \qquad y_{k+1}=\widetilde y_{k+1},\label{eq:ahpe-update}
\]
and satisfies the large-step condition
\[
 \lambda_{k+1}\norm{\widetilde y_{k+1}-z_k}
 \ge\frac\theta M.\label{eq:ahpe-large-step}
\]
If the initial distance to the solution set is at most \(D\), then
\[
 g(y_k)-g^*\le
 C_{\sigma,\theta}\frac{MD^3}{k^{7/2}}.\label{eq:ahpe-rate}
\]
This is the large-step A-HPE convergence theorem of
Monteiro--Svaiter~\cite{monteiro2013ahpe}.  The center
\eqref{eq:ahpe-center} depends on the same stepsize as the large-step
condition; the implementation below respects this coupling.
\end{theorem}

For a center \(z\), a stepsize \(\lambda>0\), and a regularization constant
\(\overline M\ge M\), define the second-order Taylor--proximal model
\[
 \mathfrak m_{z,\lambda}(d)
 :=\ip{\nabla g(z)}d+
 \frac12\ip{\nabla^2g(z)d}d+
 \frac{\overline M}{6}\norm d^3+
 \frac1{2\lambda}\norm d^2.\label{eq:true-taylor-prox}
\]
A vector \(d\) is a \(\widehat\sigma\)-approximate Taylor--proximal step if,
with
\[
 u_z(d):=\nabla g(z)+\nabla^2g(z)d+
             \frac{\overline M}{2}\norm d\,d,
\]
one has
\[
 \norm{\lambda u_z(d)+d}\le
 \widehat\sigma\norm d.\label{eq:ats-condition}
\]

\begin{lemma}[Certified implementation of the Taylor subroutine]
\label{lem:certified-ats}
Let \(g\) be \(\mu\)-strongly convex, \(L_g\)-smooth, and have an
\(M\)-Lipschitz Hessian.  Suppose certified approximations
\((\widetilde g_z,\widetilde H_z)\) satisfy
\[
 \norm{\widetilde g_z-\nabla g(z)}\le\delta_g,
 \qquad
 \norm{\widetilde H_z-\nabla^2g(z)}\le\delta_H,
 \qquad \delta_H\le\mu/4.\label{eq:ats-errors}
\]
Let \(d\) be the exact minimizer, using free computation, of
\[
 \ip{\widetilde g_z}d+
 \frac12\ip{\widetilde H_zd}d+
 \frac{\overline M}{6}\norm d^3+
 \frac1{2\lambda}\norm d^2.\label{eq:approx-taylor-prox}
\]
If
\[
 \lambda(\delta_g+\delta_H\norm d)
 \le\widehat\sigma\norm d,\label{eq:ats-certificate}
\]
then \(d\) satisfies \eqref{eq:ats-condition} for the true model.
Fix a prescribed stopping tolerance \(\rho_{\rm stop}>0\).  Before
solving the model, refine the gradient certificate at \(z\) until either
\[
 \norm{\widetilde g_z}+\delta_g\le\rho_{\rm stop},\label{eq:ats-center-stop}
\]
in which case \(z\) is returned as a certified small-gradient point, or
\[
 \norm{\widetilde g_z}-\delta_g\ge\rho_{\rm stop}/2.\label{eq:ats-center-away}
\]
The overlapping thresholds make this comparison terminate.  In the
second case, geometric refinement of the certified gradient and Hessian
errors produces \eqref{eq:ats-certificate}.  The number of refinements
is logarithmic in the displayed parameters and the requested tolerance.
\end{lemma}

\begin{proof}
The optimality condition of \eqref{eq:approx-taylor-prox} is
\[
 \widetilde g_z+\widetilde H_zd+
 \frac{\overline M}{2}\norm d\,d+\lambda^{-1}d=0.
\]
Subtracting this identity from the definition of \(u_z(d)\) gives
\[
 \norm{\lambda u_z(d)+d}
 \le\lambda(\delta_g+\delta_H\norm d),
\]
which proves the first assertion.

The center comparison terminates because, if
\(\norm{\nabla g(z)}\le\rho_{\rm stop}/2\), then
\eqref{eq:ats-center-stop} eventually holds, whereas if
\(\norm{\nabla g(z)}\ge\rho_{\rm stop}\), then
\eqref{eq:ats-center-away} eventually holds; in the overlap either
certificate is acceptable.  In the nonterminal case we therefore know
\(\norm{\nabla g(z)}\ge\rho_{\rm stop}/2\).

Let \(d_\star\) be the minimizer of the true model.  Strong convexity of
the model gives \(d_\star\ne0\), and the approximate minimizers converge
to \(d_\star\) as the certified errors decrease.  The true model equation
implies
\[
 \frac{\rho_{\rm stop}}2\le
 (L_g+\lambda^{-1})\norm{d_\star}
 +\frac{\overline M}{2}\norm{d_\star}^2,
\]
so
\[
 \norm{d_\star}\ge
 c\min\left\{
 \frac{\rho_{\rm stop}}{L_g+\lambda^{-1}},
 \sqrt{\frac{\rho_{\rm stop}}{\overline M}}
 \right\}.\label{eq:model-step-lower}
\]
Consequently a fixed relative fraction of \(\norm{d_\star}\) suffices in
\eqref{eq:ats-certificate}, and all requested derivative accuracies
have inverse-polynomial size.  Geometric refinement therefore has only
logarithmic cost.
\end{proof}

\begin{theorem}[Coupled \(\beta\)-line search for one A-HPE iteration]
\label{thm:coupled-beta-search}
Fix constants
\[
 0<\widehat\sigma<1,
 \qquad 0<\sigma_\ell<\sigma_u,
 \qquad \widehat\sigma+\sigma_u<\sigma_0<1,
\]
with
\(\sigma_\ell(1+\widehat\sigma)<
 \sigma_u(1-\widehat\sigma)\).
At an A-HPE state with \(A_k>0\), parameterize
\[
 \lambda(\beta)=A_k\frac{\beta^2}{1-\beta},
 \qquad
 z(\beta)=(1-\beta)y_k+\beta x_k,
 \qquad 0<\beta<1.\label{eq:beta-param}
\]
For every trial \(\beta\), call Lemma~\ref{lem:certified-ats} with
\(\overline M=M\) and center stopping tolerance \(\bar\rho\); if that
subroutine stops at its center, return the certified center.  Otherwise
let \(d_\beta\) be the certified step.  A bisection
procedure using a logarithmic number of such calls either returns a
trial output \(y=z+d\) with the certified bound
\[
 \norm{\nabla g(y)}\le\bar\rho,\label{eq:beta-small-gradient}
\]
or returns \(\beta\), \(\lambda=\lambda(\beta)\), \(z=z(\beta)\), and
\(d=d_\beta\) such that
\[
 \frac{\sigma_\ell}{M}
 \le\lambda\norm d\le
 \frac{\sigma_u}{M}.\label{eq:beta-large-step-window}
\]
The returned \((\lambda,z)\) obeys the coupled A-HPE relations
\eqref{eq:ahpe-a}--\eqref{eq:ahpe-center}.  When \(A_k=0\), the center is
independent of the stepsize; a fixed-center search produces the first
valid step, after which \(A_1=\lambda_1>0\) and the displayed
parameterization applies.
\end{theorem}

\begin{proof}
This is the \(d=2\), \(h\equiv0\) specialization of the implementable
A-HPE line search of Jiang--Wang--Zhang~\cite{jiang2021tensor}.  Their
parameterization is exactly \eqref{eq:beta-param}; in particular, the
acceleration center is recomputed at every bisection point rather than
fixed before choosing \(\lambda\).  Their Theorem~4.1, with
\[
 \alpha_-:=\frac{2\sigma_\ell}{M+\overline M}
 =\frac{\sigma_\ell}{M},
 \qquad
 \alpha_+:=\frac{2\sigma_u}{M+\overline M}
 =\frac{\sigma_u}{M},
\]
gives the alternative \eqref{eq:beta-small-gradient} or
\eqref{eq:beta-large-step-window}, and a logarithmic number of
bisection calls, for any approximate Taylor subroutine satisfying
\eqref{eq:ats-condition}.

Lemma~\ref{lem:certified-ats} supplies such a true-model subroutine at
every trial using the simulated oracle.  The quantity
\(\lambda\norm d\) is computed from the accepted model minimizer, so the
window decision has no additional oracle uncertainty.  Run the external
line search with the internal stopping threshold \(\bar\rho/2\).  Its small-gradient alternative is implemented with a two-threshold
certificate: refine a gradient approximation at \(y\) until either
\(\norm{\widetilde g_y}+\delta_y\le\bar\rho\), in which case we stop,
or \(\norm{\widetilde g_y}-\delta_y\ge\bar\rho/2\), in which case the
window bisection continues.  The overlap makes this decision terminate.
If the line-search theorem reaches its small-residual alternative, the
true norm is at most \(\bar\rho/2\), so the upper certificate is obtained
after finitely many refinements; no equality-at-the-boundary loop is
possible.  The line-search theorem and \eqref{eq:model-step-lower} imply that all
required derivative tolerances have inverse-polynomial size.  Thus the
bisection and all derivative refinements cost only logarithmically many
base-oracle calls.

For completeness, when \(A_k=0\) the center is fixed at \(x_0=y_0\).
For the exact Taylor model write
\[
 \phi(d)=\ip{\nabla g(z)}d+
 \frac12\ip{\nabla^2g(z)d}d+
 \frac{M}{6}\norm d^3,
 \qquad
 d(\lambda)=\argmin_d\left\{\phi(d)+
 \frac{\norm d^2}{2\lambda}\right\}.
\]
Monotonicity of \(\partial\phi\) implies that
\(\|d(\lambda)\|\) is nondecreasing in \(\lambda\): applying
monotonicity to the optimality equations at
\(0<\lambda_1<\lambda_2\) rules out
\(\|d(\lambda_1)\|>\|d(\lambda_2)\|\) by Cauchy--Schwarz.
If the stopping test does not hold, the unconstrained model minimizer is
nonzero.  Hence \(\lambda\|d(\lambda)\|\) is continuous, tends to zero
as \(\lambda\downarrow0\), and tends to infinity as
\(\lambda\uparrow\infty\).  The same certified ATS and geometric
refinement therefore provide the first window step by ordinary
fixed-center bisection.
\end{proof}

\begin{lemma}[Certified coupled A-HPE step]
\label{lem:robust-ahpe-step}
Let \(g\) be \(\mu\)-strongly convex, \(L_g\)-smooth, and have an
\(M\)-Lipschitz Hessian.  Suppose certified gradient and Hessian
approximations with arbitrary tolerances are available at logarithmic
base-oracle cost.  Fix a current contraction radius \(D\).  One coupled
line-search call either returns a point at distance at most \(D/2\) from
\(z^\star\), or returns a tuple generated at the center
\eqref{eq:ahpe-center} satisfying the A-HPE relative condition
\eqref{eq:ahpe-relative} with parameter \(\sigma_0\) and
\[
 \lambda\norm{y-z}\ge\frac{\sigma_\ell}{M}.\label{eq:certified-large-step}
\]
The total base-oracle cost is \(\widetilde O(1)\).
\end{lemma}

\begin{proof}
Run Theorem~\ref{thm:coupled-beta-search} with
\(\bar\rho=\mu D/2\).  Its small-gradient outcome at \(y\) gives
\(\norm{y-z^\star}\le\norm{\nabla g(y)}/\mu\le D/2\).
Otherwise put \(y=z+d\).  The true Taylor-model residual and Hessian
Lipschitzness imply
\[
 \begin{aligned}
 \norm{\lambda\nabla g(y)+d}
 &\le \norm{\lambda u_z(d)+d}
 +\lambda\norm{\nabla g(y)-u_z(d)}\\
 &\le(\widehat\sigma+\sigma_u)\norm d,
 \end{aligned}\label{eq:true-ahpe-residual}
\]
because
\[
 \norm{\nabla g(y)-u_z(d)}
 \le\frac{M+\overline M}{2}\norm d^2=M\norm d^2
\]
and \(\lambda\norm d\le\sigma_u/M\).

Choose a fixed slack \(\eta>0\) such that
\[
 (\widehat\sigma+\sigma_u+\eta)^2+\eta^2<\sigma_0^2<1.\label{eq:ahpe-slack}
\]
Obtain a certified gradient approximation \(v\) at \(y\) with error
\(\delta_y\) refined until
\[
 \lambda\delta_y\le\eta\norm d,
 \qquad
 \frac{\lambda\delta_y^2}{\mu}\le\eta^2\norm d^2.\label{eq:output-gradient-tol}
\]
By Lemma~\ref{lem:gradient-enlargement},
\(v\in\partial_{\epsilon_h}g(y)\) for
\(\epsilon_h=\delta_y^2/(2\mu)\).  Equations
\eqref{eq:true-ahpe-residual}--\eqref{eq:output-gradient-tol} give
\[
 \norm{\lambda v+d}^2+2\lambda\epsilon_h
 \le\sigma_0^2\norm d^2,
\]
which is \eqref{eq:ahpe-relative}.  The lower bound in
\eqref{eq:beta-large-step-window} is
\eqref{eq:certified-large-step}.  The coupled bisection takes a
logarithmic number of Taylor-subroutine calls, and each subroutine and
final-gradient refinement has logarithmic base-oracle cost.  Their
product is absorbed by \(\widetilde O(1)\).
\end{proof}

\begin{corollary}[Robust simulated second-order contraction]
\label{cor:robust-contraction}
Under the hypotheses of Lemma~\ref{lem:robust-ahpe-step}, a point at
distance at most \(D\) from \(z^\star\) can be contracted by a fixed factor
in
\[
 \widetilde O\!\left(
 1+\left(\frac{MD}{\mu}\right)^{2/7}
 \right)\label{eq:robust-contraction}
\]
base-oracle calls.  If \(M=0\), the Hessian is constant and one certified
Newton solve gives the contraction at logarithmic base-oracle cost.  Repeated restarts reach any final distance with only
logarithmic additional dependence on the requested tolerance.
\end{corollary}

\begin{proof}
If \(M=0\), \(g\) is a strongly convex quadratic.  Certified refinement of
one gradient and the constant Hessian, followed by a free linear solve,
proves the stated contraction.  Assume henceforth that \(M>0\) and run
the full recurrence
\eqref{eq:ahpe-a}--\eqref{eq:ahpe-update}.  At every iteration use
Lemma~\ref{lem:robust-ahpe-step}; hence every nonterminal tuple is a
valid A-HPE tuple at the correctly coupled center and satisfies the
large-step condition with \(\theta=\sigma_\ell\).  Theorem
\ref{thm:large-step-ahpe-blackbox} gives
\(g(y_k)-g^*\le CMD^3/k^{7/2}\).  Taking
\(k=O((MD/\mu)^{2/7})\) makes this at most a sufficiently small multiple
of \(\mu D^2\), and strong convexity contracts the distance.  The
small-gradient alternative contracts immediately.  Geometric restarts
prove the final statement.
\end{proof}

The central point is that \eqref{eq:dual-Hlip} has no polynomial factor in
\(\kap\).  Jointly solving the saddle as a generic strongly monotone
equation loses this structure and would lead to the weaker exponent
\(2/5\); eliminating the constant-conditioned primal block exposes an
ordinary strongly convex problem.  Lemma~\ref{lem:dual-oracle} and
Corollary~\ref{cor:robust-contraction} show that the optimal \(2/7\)
contraction remains implementable even though both reduced-dual gradients
and Hessians are simulated.

\section{Warm-started proximal saddle solves}\label{sec:warm-solve}
We first record the contraction consequence of the certified A-HPE construction above.

\begin{lemma}[Accelerated second-order contraction]\label{lem:convex-second-order}
Let \(g\) be \(\mu\)-strongly convex and have an \(M\)-Lipschitz Hessian.  From a point at distance at most \(D\) from its minimizer, a second-order method can reduce the distance by a fixed constant factor in
\[
 \tO\!\left(1+\left(\frac{MD}{\mu}\right)^{2/7}\right)\label{eq:convex-contract}
\]
queries.  Repeated restarts reach any final distance \(\tau>0\) with the same polynomial term and only logarithmic additional dependence on \(D/\tau\).  Independently, accelerated first-order optimization gives a contraction in \(\tO(\sqrt{L_g/\mu})\) queries when \(g\) is \(L_g\)-smooth.
\end{lemma}

\begin{proof}
Corollary~\ref{cor:robust-contraction} gives
\[
 g(y_T)-g^*\le\tO\!\left(\frac{M D^3}{T^{7/2}}\right).
\]
Strong convexity converts a target gap \(c\mu D^2\) into constant-factor distance contraction, giving \(T=\tO((MD/\mu)^{2/7})\).  Restarting over geometrically decreasing radii makes the sum of polynomial costs geometric.  The first-order statement is the usual accelerated strongly-convex rate.
\end{proof}

Apply Lemma~\ref{lem:convex-second-order} to \(\psi_x\).  Lemmas~\ref{lem:dual-curvature} and \ref{lem:dual-oracle} give the following.

\begin{proposition}[Warm-started proximal saddle solve]\label{prop:warm-solve}
Suppose \(y^{\rm in}\) satisfies \(\norm{y^{\rm in}-v_x}\le D\).  For every target joint accuracy \(\tau>0\), the saddle \(z_x=(u_x,v_x)\) can be approximated within \(\tau\) using
\[
 \boxed{
 Q_{\rm sad}(D,\tau)
 \le\tO\!\left(
 1+\min\left\{\sqrt\kap,
 \left(\frac{\rho D}{\mu}\right)^{2/7}\right\}
 \right)
 }\label{eq:saddle-solve}
\]
original bundled calls.  Dependence on \(\tau\) is logarithmic.
\end{proposition}

\begin{proof}
The reduced dual is \(\mu\)-strongly convex, \(2L\)-smooth, and has an \(O(\rho)\)-Lipschitz Hessian.  Lemma~\ref{lem:dual-oracle} supplies certified approximate gradients and Hessians at logarithmic original-oracle cost, so Corollary~\ref{cor:robust-contraction} gives the second-order branch.  The accelerated first-order alternative gives the other branch.  Once \(v_x\) is approximated, Lemma~\ref{lem:primal-elim} and a constant-conditioned primal solve recover \(u_x\).  All primal refinements and the final constant-conditioned primal recovery change only logarithmic factors.
\end{proof}

For consecutive outer centers, let \(\widetilde v_{t-1}\) denote the
certified reduced-dual approximation returned at \(x_{t-1}\), and set
\[
 \tau_{t-1}:=\norm{\widetilde v_{t-1}-v_{x_{t-1}}},
 \qquad
 D_t:=\norm{\widetilde v_{t-1}-v_{x_t}}.
\]
Thus \(D_t\) is the initial distance from the warm start
\(\widetilde v_{t-1}\) to the reduced-dual minimizer at the new center.
Lemma~\ref{lem:sens} gives
\[
 D_t\le \tau_{t-1}+2\sqrt\kap\norm{x_t-x_{t-1}}.\label{eq:warm-distance}
\]
This is the interface that will be amortized rather than bounded separately at every iteration.

\section{Outer algorithms and proof of the upper bound}\label{sec:upper-proof}
We need two robust outer facts.  Their standard proofs follow from the descent lemma for gradient descent and the cubic Taylor upper model for cubic regularization.

\begin{lemma}[Robust outer bounds]\label{lem:outer}
Let \(E\) be bounded below with initial gap at most \(\Delta\).
\begin{enumerate}[label=(\roman*)]
\item If \(\nabla E\) is \(L_E\)-Lipschitz and each queried gradient has error at most a sufficiently small constant multiple of \(\eps\), inexact gradient descent reaches \(\norm{\nabla E}\le\eps\) in
\[
 N_1=O(\Delta L_E\eps^{-2})
\]
iterations, with steps \(s_t\) satisfying
\[
 \sum_t\norm{s_t}^2\le C\frac\Delta{L_E}.\label{eq:gd-step-sum}
\]
\item If \(\nabla^2E\) is \(M_E\)-Lipschitz, and gradient and Hessian errors are bounded by small constant multiples of \(\eps\) and \(\sqrt{M_E\eps}\), respectively, inexact cubic regularization reaches \(\norm{\nabla E}\le\eps\) in
\[
 N_2=O(\Delta\sqrt{M_E}\eps^{-3/2})
\]
iterations, with
\[
 \sum_t\norm{s_t}^3\le C\frac\Delta{M_E}.\label{eq:cubic-step-sum}
\]
\end{enumerate}
\end{lemma}

\begin{proof}
For gradient descent, use the step \(s_t=-\widetilde g_t/L_E\) and the
smoothness inequality.  When the gradient error is at most a sufficiently
small multiple of \(\varepsilon\), every nonterminal step decreases the
objective by \(c\|\nabla E(x_t)\|^2/L_E\) and also satisfies
\(\|s_t\|^2\le C(E(x_t)-E(x_{t+1}))/L_E\).  Summation gives both claims.
For cubic regularization, minimize the standard upper model
\(\langle\widetilde g_t,s\rangle+\langle\widetilde H_ts,s\rangle/2+
M_E\|s\|^3/6\).  The Taylor remainder and the stated oracle errors imply
that every nonterminal step decreases the objective by
\(cM_E\|s_t\|^3\) and that the next gradient is
\(O(M_E\|s_t\|^2)\) up to the prescribed error.  Summing the decreases and
using \(\|\nabla E\|>\varepsilon\) before termination yields the iteration
and cubic-step bounds.
\end{proof}

At an approximate saddle \(\widetilde z=(\widetilde u,\widetilde v)\) for center \(x\), define the plug-in envelope gradient and Hessian by
\[
 \widetilde g_{\calE}=2L(x-\widetilde u),
 \qquad
 \widetilde H_{\calE}=2LI-4L^2\widetilde S^{-1},\label{eq:E-plugin}
\]
where \(\widetilde S\) is the Schur complement evaluated at \(\widetilde z\).

\begin{lemma}[Envelope-oracle simulation]\label{lem:E-oracle}
If \(\norm{\widetilde z-z_x}\le\tau\), then
\[
 \norm{\widetilde g_{\calE}-\nabla\calE(x)}\le2L\tau,\label{eq:E-grad-error}
\]
\[
 \norm{\widetilde H_{\calE}-\nabla^2\calE(x)}
 \le C_E\rho(1+\kap)^2\tau.\label{eq:E-hess-error}
\]
Thus the inexact gradient and cubic tolerances in Lemma~\ref{lem:outer} can be met with a target \(\tau\) whose effect on Proposition~\ref{prop:warm-solve} is only logarithmic.
\end{lemma}

\begin{proof}
The gradient estimate is immediate.  Every exact or approximate Schur complement is at least \(LI\).  Expanding the difference of Schur complements and using \(\norm{H_{uy}}\le L\), \(\norm{H_{yy}^{-1}}\le1/\mu\), Hessian Lipschitzness, and the resolvent identity gives
\[
 \norm{\widetilde S-S_x}\le C\rho(1+\kap)^2\tau.
\]
Applying the inverse resolvent identity inside \eqref{eq:E-plugin} proves \eqref{eq:E-hess-error}.
\end{proof}

\subsection{The middle branch: envelope gradient descent}
Apply Lemma~\ref{lem:outer}(i) to \(E=\calE\).  Lemmas~\ref{lem:transfer} and \ref{lem:E-reg} give
\[
 N_1=O(A_1),
 \qquad
 \sum_t\norm{s_t}^2\le C\frac\Delta L.\label{eq:gd-data}
\]
At every outer point solve the proximal saddle to gradient-oracle accuracy \(\tau_g=c\eps/L\).  Warm-start the next reduced-dual solve at the previous approximate dual point.  Equations \eqref{eq:saddle-solve} and \eqref{eq:warm-distance}, together with \((a+b)^{2/7}\le a^{2/7}+b^{2/7}\), give for the second-order inner branch
\begin{align}
 Q_{\rm GD}
 &\le\tO\!\left(
 N_1+N_1\left(\frac{\rho\tau_g}{\mu}\right)^{2/7}
 +\left(\frac{\rho\sqrt\kap}{\mu}\right)^{2/7}
 \sum_t\norm{s_t}^{2/7}
 \right).\label{eq:gd-cost-start}
\end{align}
Holder's inequality gives
\[
 \sum_t\norm{s_t}^{2/7}
 \le N_1^{6/7}
 \left(\sum_t\norm{s_t}^2\right)^{1/7}.\label{eq:holder-gd}
\]
Substituting \eqref{eq:gd-data}, \(\mu=L/\kap\), and \eqref{eq:def-r} into \eqref{eq:gd-cost-start} yields
\[
 Q_{\rm GD}
 \le\tO\!\left(A_1(1+r^{2/7})\right),\label{eq:gd-final}
\]
because the fixed tolerance term is
\[
 A_1\left(\frac{\rho\eps}{L\mu}\right)^{2/7}
 =A_1\left(\frac r{\sqrt\kap}\right)^{2/7}
 \le A_1(1+r^{2/7}).
\]
The nontrivial middle contribution can be written explicitly as
\[
 A_1r^{2/7}
 =\Delta\rho^{2/7}\kap^{3/7}L^{3/7}\eps^{-12/7}.\label{eq:middle-rate}
\]
Using the first-order branch in Proposition~\ref{prop:warm-solve} instead gives
\[
 Q_{\rm GD,FO}\le\tO(\sqrt\kap A_1).\label{eq:FO-cap}
\]

\subsection{The high-accuracy branch: envelope cubic regularization}
Apply Lemma~\ref{lem:outer}(ii) to \(E=\calE\).  From \eqref{eq:ME},
\[
 N_2=O(A_2),
 \qquad
 \sum_t\norm{s_t}^3
 \le C\frac\Delta{\rho\kap^{3/2}}.\label{eq:cubic-data}
\]
Simulate each envelope gradient and Hessian using Proposition~\ref{prop:warm-solve} and Lemma~\ref{lem:E-oracle}.  At every outer center require the saddle accuracy
\[
 \tau_c\le
 \min\left\{
  c_g\frac{\eps}{L},
  c_H\frac{\sqrt{M_{\calE}\eps}}
                 {\rho(1+\kap)^2},
  c_0\frac{\mu}{\rho}
 \right\},\label{eq:cubic-fixed-target}
\]
with sufficiently small universal constants, using the convention
\(\mu/\rho=+\infty\) when \(\rho=0\).  The first two terms certify the
inexact-gradient and inexact-Hessian requirements of
Lemma~\ref{lem:outer}; the third is an optional stricter target.  It changes the inner work only logarithmically and makes the fixed warm-start contribution satisfy
\[
 N_2\left(\frac{\rho\tau_c}{\mu}\right)^{2/7}
 \le C N_2.
\]
The remaining, movement-dependent second-order reduced-dual contribution is bounded by
\[
 \left(\frac{\rho\sqrt\kap}{\mu}\right)^{2/7}
 \sum_t\norm{s_t}^{2/7}.
\]
Now
\[
 \sum_t\norm{s_t}^{2/7}
 \le N_2^{19/21}
 \left(\sum_t\norm{s_t}^3\right)^{2/21}.\label{eq:holder-cubic}
\]
Substitution of \eqref{eq:cubic-data} gives
\[
 Q_{\rm cubic}
 \le\tO\!\left(A_2(1+r^{1/7})\right).\label{eq:cubic-final}
\]
In particular, if \(r\le1\), this is \(\tO(A_2)\).

\subsection{Output certification and proof of Theorem~\ref{thm:main}}
At the terminal outer center \(x\), refine the proximal saddle until
\[
 \norm{\widetilde u-u_x}
 \le \frac{c\eps}{L(1+\kap)}.\label{eq:output-tol}
\]
Since \(\nabla\Phi\) is \(L(1+\kap)\)-Lipschitz and \(\nabla\Phi(u_x)=\nabla\calE(x)\),
\[
 \norm{\nabla\Phi(\widetilde u)}
 \le\norm{\nabla\calE(x)}
 +L(1+\kap)\norm{\widetilde u-u_x}
 \le\eps
\]
for suitable constants in the outer stopping rule.  The final refinement changes only logarithmic factors.

The initial dual distance obeys
\[
 \norm{y_0-v_{x_0}}
 \le R_y+\kap\norm{u_{x_0}-x_0}
 \le R_y+\kap\sqrt{\Delta/L}=D_{y,0},
\]
which proves \eqref{eq:init} by Proposition~\ref{prop:warm-solve}.  Taking the minimum of \eqref{eq:gd-final}, \eqref{eq:FO-cap}, and \eqref{eq:cubic-final} proves \eqref{eq:main-upper}.  Comparing powers of \(r\) gives the piecewise form \eqref{eq:Psi}, completing the proof of Theorem~\ref{thm:main}.

\section{Randomized lower bound}\label{sec:lower}
We record a self-contained lift from an all-derivative single-level lower bound.

\begin{lemma}[Two-smoothness all-derivative lower bound]\label{lem:single-lower}
There are universal constants \(c,C>0\) such that for every \(\ell,\rho,\delta,\Delta>0\) with
\[
 T=c\Delta\min\{\ell\delta^{-2},\sqrt\rho\,\delta^{-3/2}\}\ge C,
\]
there is a dimension \(d=O(T^2\log T)\) and a smooth function \(F:\R^d\to\R\) satisfying
\[
 F(0)-\inf F\le\Delta,
 \quad
 \norm{\nabla^2F}\le\ell,
 \quad
 \mathrm{Lip}(\nabla^2F)\le\rho,\label{eq:F-class}
\]
such that every randomized local algorithm, even one receiving all derivatives at each query, needs at least \(T\) queries to find \(w\) with \(\norm{\nabla F(w)}\le\delta\) with probability at least \(2/3\).
\end{lemma}

\begin{proof}
This is the randomized all-derivative lower theorem of
Carmon--Duchi--Hinder--Sidford, after scaling the normalized hard family
so that the gradient and Hessian Lipschitz constants are at most
\(\ell\) and \(\rho\).  The theorem permits arbitrary randomized local
algorithms and an output point that need not have been queried.
\end{proof}

For \(\kap\ge6\), set
\[
 b^2=\frac{\kap-3}{2},
 \qquad
 \delta=\frac{3\eps}{2b}.\label{eq:b-delta}
\]
Apply Lemma~\ref{lem:single-lower} with \(\ell=\mu\).  Let \(g_0=\nabla F(0)\), set \(c=-g_0/(2\mu)\), and define
\[
 f(x,y)=F(y)-\mu\norm{y-bx-c}^2.\label{eq:lift}
\]

\begin{lemma}[Validity and stationarity transfer of the lift]\label{lem:lift}
The lifted function belongs to the class with parameters \((L,\mu,\rho)\), has \(y^\star(0)=0\), and satisfies
\[
 \Phi(0)-\inf\Phi\le\Delta.\label{eq:lift-gap}
\]
For every \(x\), if \(u=bx+c\), then
\[
 \norm{\nabla F(u)}
 \le\frac{3}{2b}\norm{\nabla\Phi(x)}.\label{eq:lift-stationarity}
\]
One query to the lifted bundled oracle can be simulated by one all-derivative query to \(F\).
\end{lemma}

\begin{proof}
The \(yy\) Hessian is \(\nabla^2F-2\mu I\preceq-\mu I\).  The quadratic part has joint Hessian norm \(2\mu(1+b^2)\), so
\[
 \norm{\nabla^2f}\le\mu+2\mu(1+b^2)=\mu(3+2b^2)=L.
\]
Only \(\nabla^2F(y)\) varies, hence the Hessian Lipschitz constant is at most \(\rho\).  The choice of \(c\) makes \(y^\star(0)=0\).

For every \(x\), evaluating at \(y=bx+c\) gives \(\Phi(x)\ge\inf F\).  Conversely, for every \(y\), choosing
\[
 x=\frac{y-c-\nabla F(y)/(2\mu)}b
\]
makes \(y\) the unique inner maximizer and gives \(\Phi(x)=F(y)-\norm{\nabla F(y)}^2/(4\mu)\le F(y)\).  Hence \(\inf\Phi=\inf F\) and \eqref{eq:lift-gap} follows.

At the inner maximizer,
\[
 \nabla F(y^\star)=2\mu(y^\star-bx-c),
 \qquad
 \nabla\Phi(x)=b\nabla F(y^\star).
\]
Also \(y^\star-u=\nabla F(y^\star)/(2\mu)\).  Since \(\nabla F\) is \(\mu\)-Lipschitz,
\[
 \norm{\nabla F(u)}
 \le\frac32\norm{\nabla F(y^\star)}
 =\frac{3}{2b}\norm{\nabla\Phi(x)}.
\]
All remaining terms in the lifted oracle are known quadratics.
\end{proof}

Lemma~\ref{lem:single-lower} and \eqref{eq:lift-stationarity} imply
\begin{align}
 N
 &\ge\Omega\!\left(
 \Delta\min\left\{\mu(b/\eps)^2,
 \sqrt\rho\,(b/\eps)^{3/2}\right\}
 \right)\\
 &=\Omega\!\left(
 \Delta\min\left\{L\eps^{-2},
 \sqrt\rho\,\kap^{3/4}\eps^{-3/2}\right\}
 \right).\label{eq:all-derivative-baseline}
\end{align}
For \(1\le\kap<6\), the decoupled lift \(f(x,y)=F(x)-\mu\norm y^2/2\) gives the same bound up to a universal constant.  Thus \eqref{eq:all-derivative-baseline} is an independently established baseline: it proves the desired lower rate for \(r\le1\) and gives \(\Omega(A_1)\) for \(r\ge1\).

\section{Exactly-flat scalar edge and probabilistic tool}
\subsection{A convex exactly-flat truncated cubic}
Fix a nondecreasing \(C^\infty\) flat-core profile
\(\chi_{\rm fc}:[0,\infty)\to[0,1]\) satisfying
\[
 \chi_{\rm fc}(u)=0\quad(0\le u\le1/2),
 \qquad
 \chi_{\rm fc}(u)=1\quad(u\ge1).
\]
For \(0<\delta_{\rm fc}\le\Delta\), define, for \(t\ge0\),
\[
 p_{\delta_{\rm fc},\Delta}(t)=
 \begin{cases}
   t^2\chi_{\rm fc}(t/\delta_{\rm fc}),&0\le t\le\Delta,\\
   2\Delta t-\Delta^2,&t\ge\Delta,
 \end{cases}
 \qquad
 \widehat g_{\delta_{\rm fc},\Delta}(t)
 =\int_0^{|t|}p_{\delta_{\rm fc},\Delta}(s)\,ds.\label{eq:flat-cubic}
\]

\begin{lemma}[Flat-core scalar edge]\label{lem:flat-cubic}
There is a universal \(C_{\rm fc}\) such that \(\widehat g_{\delta_{\rm fc},\Delta}\) is even, convex, and twice continuously differentiable, and:
\begin{enumerate}
\item its value, first derivative, and second derivative vanish on \(|t|\le\delta_{\rm fc}/2\);
\item for \(t\ge0\),
\[
 0\le \widehat g'_{\delta_{\rm fc},\Delta}(t)\le t^2;\label{eq:p-upper}
\]
\item its first and second derivatives agree with those of the usual truncated cubic
\[
 g_\Delta(t)=
 \begin{cases}
 |t|^3/3,&|t|\le\Delta,\\
 \Delta t^2-\Delta^2|t|+\Delta^3/3,&|t|>\Delta
 \end{cases}
\]
whenever \(|t|\ge\delta_{\rm fc}\);
\item
\[
 0\le \widehat g''_{\delta_{\rm fc},\Delta}(t)\le C_{\rm fc}\Delta,
 \qquad
 \mathrm{Lip}(\widehat g''_{\delta_{\rm fc},\Delta})\le C_{\rm fc}.\label{eq:flat-cubic-regularity}
\]
\end{enumerate}
The values outside the core may differ from \(g_\Delta\) by an additive constant of order \(\delta_{\rm fc}^3\); this does not affect the oracle derivatives or the minimizer equations.
\end{lemma}

\begin{proof}
For \(t>0\), \(\widehat g'=p_{\delta_{\rm fc},\Delta}\).  The profile is nonnegative and nondecreasing because \(\chi_{\rm fc},\chi_{\rm fc}'\ge0\), so the integral is convex.  It is identically zero on \([0,\delta_{\rm fc}/2]\).  On \([\delta_{\rm fc},\Delta]\), \(p(t)=t^2\), and on \([\Delta,\infty)\) it equals the derivative of the quadratic continuation of the truncated cubic.  The pieces match in value and first derivative at \(\Delta\).  Differentiating \(t^2\chi_{\rm fc}(t/\delta_{\rm fc})\) twice on the fixed transition interval gives the bounds in \eqref{eq:flat-cubic-regularity}; outside that interval the corresponding slopes are bounded by the same universal constant.  Finally, \(p(t)\le t^2\) also for \(t\ge\Delta\), since
\(t^2-(2\Delta t-\Delta^2)=(t-\Delta)^2\).
\end{proof}

\subsection{A generic Fenchel distance-to-stationarity transfer}
\begin{lemma}[Fenchel transfer]\label{lem:fenchel-distance-transfer}
Let \(H:\mathbb R^d\to\mathbb R\) be \((\mu+c)\)-strongly convex with
minimizer \(w^\star\), let \(h(y)=H(y)-c\|y\|^2/2\), and choose
\(\beta,c>0\) and \(\lambda^2=\beta c\).  Define
\[
 f(x,y)=\frac\beta2\|x\|^2+\lambda\langle x,y\rangle-h(y)+m,
 \qquad m=\min H,
\]
and \(\Phi(x)=\max_y f(x,y)\).  If \(w=-\beta x/\lambda\) and \(y_x\) is
the inner maximizer, then
\[
 \nabla\Phi(x)=\lambda(y_x-w),\qquad
 \nabla H(y_x)=c(y_x-w),\label{eq:fenchel-identities-core}
\]
\[
 \|\nabla\Phi(x)\|
 \ge \frac{\lambda}{1+c/(\mu+c)}\|w-w^\star\|.\label{eq:fenchel-distance-core}
\]
Moreover, \(\Phi\) is \(\beta\)-strongly convex, is minimized at
\(x^\star=-\lambda w^\star/\beta\), and has minimum zero.
\end{lemma}
\begin{proof}
The inner equation is \(\lambda x=\nabla h(y_x)\).  Since
\(\lambda x=-cw\) and \(\nabla h=\nabla H-cy\), this gives
\eqref{eq:fenchel-identities-core}.  Strong convexity of \(H\) yields
\(\|y_x-w^\star\|\le c\|y_x-w\|/(\mu+c)\); the triangle inequality and
\(\|y_x-w\|=\|\nabla\Phi(x)\|/\lambda\) give
\eqref{eq:fenchel-distance-core}.  The remaining claims follow by
completing the square and observing that a pointwise maximum of
\(\beta\)-strongly convex functions is \(\beta\)-strongly convex.
\end{proof}
\subsection{Adaptive Haar concentration}
\begin{lemma}[Adaptive Haar concentration for a virtual filtration]
\label{lem:adaptive-haar-side}
Let \(V=(v_1,\ldots,v_m)\) be a Haar-uniform orthonormal frame in
\(\R^d\), independent of a sigma-field \(\mathcal G\).  Before round \(n\),
a prefix \(v_1,\ldots,v_{k_n}\) has been disclosed and
\[
 \mathcal F_n=\sigma(\mathcal G,v_1,\ldots,v_{k_n},\mathcal U_n),
\]
where the auxiliary pre-query virtual history \(\mathcal U_n\) is
measurable with respect to \(\mathcal G\) and the disclosed columns only.
The disclosure protocol is predictable in the following strong sense:
a query is selected before any fresh column is disclosed, and after its
response the protocol may disclose the next column in full, without
conditioning disclosure on any event involving that column.  Thus no
``no-hit'' information about an undisclosed column enters
\(\mathcal F_n\).  Suppose the procedure chooses at most \(J\) candidate vectors
\(z_{n,1},\ldots,z_{n,J}\) that are \(\mathcal F_n\)-measurable and satisfy
\(\|z_{n,j}\|\le R\).  It may also produce at most \(J\) final vectors with
the same measurability and norm bound.  Then
\[
 \Pr\left\{
  \max_{n,j}\max_{i>k_n}|\langle v_i,z_{n,j}\rangle|>s
 \right\}
 \le 2J(N+1)m
 \exp\!\left(-c(d-m)\frac{s^2}{R^2}\right).
\]
In particular, the probability is at most \(\delta_{\rm prob}\) if
\[
 d-m\ge C\frac{R^2}{s^2}
 \log\!\left(\frac{2J(N+1)m}{\delta_{\rm prob}}\right).
\]
\end{lemma}

\begin{proof}
Condition on \(\mathcal F_n\).  By the predictable disclosure rule,
no variable in \(\mathcal F_n\) depends on an undisclosed column except
through the already disclosed span; in particular, the history contains
no negative information of the form ``the next column was not hit.''
The remaining columns are therefore Haar-uniform in the orthogonal
complement of that span, whose dimension is at least \(d-m+1\).  For every fixed \(\mathcal F_n\)-measurable vector of norm at
most \(R\), the spherical tail bound gives
\[
 \Pr\{|\langle v_i,z\rangle|>s\mid\mathcal F_n\}
 \le2\exp(-c(d-m)s^2/R^2).
\]
Average over the past and take a union bound over candidate--column
pairs.  Random activation times are harmless because the pre-query
filtration is generated only by the disclosed prefix and the predictable
virtual history; it never conditions on an event involving an
undisclosed column.
\end{proof}

\begin{remark}
The measurability restriction on \(\mathcal U_n\) is essential.  A lemma
that merely permits arbitrary ``preceding responses'' is false, because
such responses may already encode unrevealed columns.  The
packet-suppressed virtual transcripts below are constructed precisely so
that the corrected hypothesis holds.
\end{remark}

\section{Private-handle lower bound for \texorpdfstring{\(r\ge1\)}{r >= 1}}
\label{sec:private-handle}
We use a private handle so that the first selected path direction of a
future packet is absent from every inactive value--gradient--Hessian
response.  A public orientation-free quadratic amplifier maps that
handle to the first path direction only after activation.  The same
amplifier reduces the physical seed to \(\Theta(\mu D)\) while preserving
an \(\Omega(D)\) hidden tail, which is the scale needed for the relay
Schur-complement argument.

Throughout this section all universal constants are selected in the order
in which they appear.  Let \(q_0\ge1\) be the universal positive-integer
minimum packet length required by the serial product theorem below.  If
\[
 c_q\min\{r^{2/7},\sqrt\kap\}<q_0,
\]
then the target multiplier is bounded by a universal constant, and the
independent baseline \eqref{eq:all-derivative-baseline} already gives the
required order.  The private-handle construction is used only in the
complementary, nondegenerate regime, where its packet length satisfies
\(q\ge q_0\).

In this remaining regime, \(\kap\) is larger than a universal constant.  We
use a standard constant-slack convention and carry out the construction with
internal budgets
\[
 L_\circ=c_L L,\qquad \mu_\circ=C_\mu\mu,
 \qquad \kap_\circ=L_\circ/\mu_\circ\asymp\kap,\label{eq:slack-budgets}
\]
where \(c_L>0\) is sufficiently small and \(C_\mu>1\) is sufficiently large.
All Hessian norms then fit below \(L\), while the final dual strong-concavity
modulus is at least \(\mu\).  Replacing \((L,\mu,\kap,r)\) by the internal
quantities changes every displayed rate only by universal constants.  We
suppress the circles below.

\begin{lemma}[Feasible universal constant hierarchy]
\label{lem:private-constant-hierarchy}
All constants in the private-handle packet, paired shell, Fenchel lift,
and dual relay can be fixed once and for all, independently of
\((L,\mu,\rho,\Delta,\eps,\kap,r,T,q)\), so that every displayed class,
curvature, gate, and gap inequality used below holds simultaneously.
One admissible order of choice is
\[
 C_\omega,C_A,C_R,C_{\rm col},K_0,C_h,C_a,K_D,c_\sigma,C_\mu
 \quad\text{first},
 \qquad
 c_L,c_g,c_q,c_T
 \quad\text{afterwards}.\label{eq:constant-choice-order}
\]
More precisely, after the profile functions are fixed, the constants may
be chosen so that
\[
 \begin{array}{c|c}
 \text{requirement}&\text{reserved budget}\\
 \hline
 \text{packet and Fenchel Hessian norms}&L/16\ \text{each}\\
 \text{relay quadratic joint Hessian}&L/8\\
 \text{all gate-product Hessian terms}&L/8\\
 \text{packet and shell Hessian-Lipschitz terms}&\rho/8\ \text{each}\\
 \text{all gate-product third-derivative terms}&\rho/4\\
 \text{dual strong-concavity remainder}&2\mu\\
 \text{initial-gap and interpolation shifts}&\Delta/2.
 \end{array}\label{eq:constant-budget-table}
\]
Consequently the final family is jointly \(L\)-smooth, has
\(\rho\)-Lipschitz Hessian, is \(\mu\)-strongly concave in all dual
variables, and has initial value gap at most \(\Delta\).
\end{lemma}

\begin{proof}
Fix first all scalar profiles; their derivative norms are now numerical
constants.  Choose \(C_\omega\) large enough for the bulk-tail estimate,
then \(C_A\) large enough that the bounded seed uses at most one eighth of
the packet strong-convexity margin.  Choose \(C_R\) so that the active
minimizer and every chamber used in the information proof lie strictly
inside the identity region of the common shell.  Next choose
\(C_{\rm col}\) to dominate, with a factor four, the operator norm of all
cutoff-product negative Hessian terms on the shell annulus.  Choose
\(K_0\) so that the interpolation coefficient satisfies
\(-C_BE\le \Delta f-K_0E\le-E\), and then choose \(C_h\) larger than the
universal mixed-block Schur-complement constant in
\eqref{eq:handle-young}.  Finally choose \(C_a\) and \(K_D\) so that the
relay has the required geometric width, the stationarity force has a
fixed margin above \(\eps\), and the packet minimizer remains in every
linear core.

Choose next the universal relay coefficient \(c_\sigma\) large enough
that, in the internal scale, the positive \(s\)-curvature of the relay
quadratic dominates by a factor four the largest possible negative
curvature contributed by the radial gate and the progress reward.  Only
a universal numerical lower bound on \(c_\sigma\) is required.  After all
large constants, including \(c_\sigma\), have been fixed, choose the
internal smoothness fraction \(c_L\) sufficiently small and the internal
curvature multiplier \(C_\mu\) sufficiently large in
\eqref{eq:slack-budgets}.  These choices can be made simultaneously so
that
\[
 4\mu\le \sigma_{\rm rel}=c_\sigma L_\circ/a_{\rm rel}^2
 \quad\text{and}\quad
 \sigma_{\rm rel}(1+a_{\rm rel}^2)\le L/8.
\]
Indeed \(1\le a_{\rm rel}^2\le C\kap_\circ\), so the first inequality
follows once the internal strong-concavity modulus is a sufficiently
large universal multiple of the target \(\mu\), whereas the second follows
from \(c_L\le c/c_\sigma\).  Thus the relay coefficient may be large on
the internal scale while its entire joint Hessian still occupies only
the reserved external \(L/8\) budget.  Choose \(c_g\) small enough that the
packet's nonlinear edge Hessian and third derivative occupy their
assigned fractions in \eqref{eq:constant-budget-table}.  Choose \(c_q\) small enough that
\(q^2\le c\kap_\circ\) and \(q^{7/2}\le cr\) force respectively
\(\nu q^2\le L/16\) and \(\rho_q\le\rho/8\).  With these choices the four
third-derivative ratios in \eqref{eq:handle-relay-ratios} are bounded by
the reserved \(\rho/4\) budget.  The lower-right dual Hessian retains the
margin \(C_h\mu\), while Young's inequality consumes less than
\((C_h-2)\mu\); together with the relay upper-left block this leaves at
least \(\mu I\).  Last choose \(c_T\) so that
\(CTE\le\Delta/2\) and all one-time initialization shifts fit in the
remaining half of the gap budget.  Every choice only tightens a later
small constant and no step asks a previously fixed large constant to be
changed.  This proves simultaneous feasibility.
\end{proof}

\subsection{A public amplifier with a private handle}
Let \(\mathsf A,\mathsf B,\mathsf W\) be mutually orthogonal Euclidean
spaces, with \(\dim\mathsf A=\dim\mathsf B\), and let
\(C:\mathsf A\to\mathsf B\) be a fixed public isometry.  Select a hidden
unit handle \(h\in\mathsf A\), put
\[
 v_1:=Ch\in\mathsf B,\label{eq:v1-handle}
\]
and choose hidden orthonormal vectors
\(v_2,\ldots,v_m\in\mathsf W\).  The packet variable is
\(y=(a,b,u)\in\mathsf A\oplus\mathsf B\oplus\mathsf W\) and we write
\(w=b+u\).

Define the public quadratic form
\[
 Q_q(a,b,u)
 :=\frac\nu2\left(
 2\norm a^2+(q+1)\norm b^2-2\sqrt q\,\ip{Ca}{b}+\norm u^2
 \right),\label{eq:public-amplifier}
\]
where \(\nu\asymp\mu\) is a sufficiently large universal multiple of the
required strong-convexity scale.

\begin{lemma}[Spectrum and no-tail Krylov leakage]
\label{lem:amplifier-spectrum}
The Hessian of \(Q_q\) has spectrum contained in \([\nu,\nu(q+2)]\).  More
precisely, on every paired two-dimensional space generated by
\((e,0,0)\) and \((0,Ce,0)\) its matrix divided by \(\nu\) is
\[
 \begin{bmatrix}2&-\sqrt q\\-\sqrt q&q+1\end{bmatrix},\label{eq:amplifier-matrix}
\]
whose eigenvalues are \(1\) and \(q+2\).  Moreover, after \(h\) is known, the
Krylov space generated by the public Hessian from \((h,0,0)\) is exactly
contained in
\[
 \mathrm{span}\{(h,0,0),(0,v_1,0)\};\label{eq:amplifier-krylov}
\]
it contains none of \(v_2,\ldots,v_m\).
\end{lemma}

\begin{proof}
The characteristic polynomial of \eqref{eq:amplifier-matrix} is
\((\lambda-1)(\lambda-q-2)\).  On \(\mathsf W\) the Hessian is \(\nu I\).
The public Hessian maps the two vectors in
\eqref{eq:amplifier-krylov} into their span and has no block connecting
that span to \(\mathsf W\), proving the last assertion.
\end{proof}

The importance of Lemma~\ref{lem:amplifier-spectrum} is that the full
quadratic carrier is public but does not choose \(h\).  Before the active
seed appears, neither the value, gradient, nor Hessian of
\eqref{eq:public-amplifier} contains the selected handle.  After the seed
reveals \(h\), the public map gives \(v_1=Ch\), but no tail direction.

\subsection{A low-seed exactly-flat distance packet}
Fix a packet Hessian-Lipschitz budget \(\rho_q>0\).  Use the scalar edge
\(\widehat g_{\delta_q,\omega}\) from
Lemma~\ref{lem:flat-cubic}.  Choose a sufficiently small universal
\(c_g>0\) and put
\[
 \bar\lambda:=\frac\nu{c_g\rho_q},\qquad
 \omega:=C_\omega q^2\bar\lambda,\qquad
 \delta_q:=c_\delta\omega,\label{eq:handle-scales}
\]
where \(C_\omega\) is large and \(c_\delta\) is small.  Let
\[
 \Gamma:=\omega^2,\qquad
 \eta:=\frac{2\Gamma}{\sqrt q},\qquad
 m:=\lceil C_mq\rceil,\label{eq:handle-seed}
\]
with \(C_m\) fixed before \(C_\omega\).  Define
\[
 \boxed{
 \begin{aligned}
 \mathcal H_{\mathcal V}^{\rm lin}(a,b,u)
 :=c_g\rho_q\Bigg[&
 \sum_{i=1}^{m-1}
 \widehat g_{\delta_q,\omega}
  \bigl(\ip{v_i-v_{i+1}}{b+u}\bigr)\\
 &+\frac{\bar\lambda}{2}
 \left(2\norm a^2+(q+1)\norm b^2
 -2\sqrt q\,\ip{Ca}{b}+\norm u^2\right)
 -\eta\ip h a\Bigg].
 \end{aligned}}\label{eq:handle-packet}
\]
The tuple \(\mathcal V=(h,v_2,\ldots,v_m)\) is private, while \(C\) is public.
The seedless function, obtained by deleting the last term, is denoted
\(\mathcal H_{\mathcal V}^{0,\rm lin}\).

\begin{theorem}[Private-handle low-seed packet]
\label{thm:private-handle-packet}
After the universal constants above are fixed, the following assertions
hold.  Define
\[
 D_q:=q^{3/2}\omega
 \asymp\frac{\nu q^{7/2}}{\rho_q}.\label{eq:handle-Dq}
\]
\begin{enumerate}[label=(\roman*)]
\item The packet is globally \(\nu\)-strongly convex,
\[
 \mathrm{Lip}(\nabla\mathcal H_{\mathcal V}^{\rm lin})
 \le C\nu q^2,
 \qquad
 \mathrm{Lip}(\nabla^2\mathcal H_{\mathcal V}^{\rm lin})
 \le\rho_q.\label{eq:handle-regularity}
\]
\item Its unique minimizer
\(y_{\mathcal V}^\star=(a^\star,b^\star,u^\star)\) obeys
\[
 \norm{y_{\mathcal V}^\star}\le CD_q,
 \qquad
 \left\|\mathrm{Proj}_{\mathrm{span}
 \{v_2,\ldots,v_{\lfloor c_0q\rfloor}\}}
 (b^\star+u^\star)\right\|\ge cD_q,\label{eq:handle-bulk}
\]
and
\[
 \mathcal H_{\mathcal V}^{\rm lin}(0)
 -\mathcal H_{\mathcal V}^{\rm lin}(y_{\mathcal V}^\star)
 \le C\nu D_q^2.\label{eq:handle-gap}
\]
\item The physical handle-seed magnitude
\[
 s_q:=c_g\rho_q\eta\label{eq:handle-physical-seed}
\]
satisfies
\[
 c\nu D_q\le s_q\le C\nu D_q.\label{eq:low-seed}
\]
Thus the seed is smaller by a factor \(\sqrt q\) than in the packet of
the scalar-edge construction above, while the bulk hidden distance
and the gap scale are unchanged.
\item Give \(h\) and hence \(v_1=Ch\) to the algorithm for free.  Against
any deterministic exact second-order algorithm using fewer than \(cq\)
responses, a resisting choice of \(v_2,\ldots,v_m\) leaves output distance
at least \(cD_q\) from \(y_{\mathcal V}^\star\).  Before activation, when
\(h\) is not given, the seedless localized packet has a public complete
transcript whenever all hidden edge arguments lie in their flat cores.
The randomized fixed-instance statement is proved later by the explicit
virtual-filtration product argument.
\end{enumerate}
\end{theorem}

\begin{proof}
The strong-convexity and regularity bounds follow from
Lemma~\ref{lem:amplifier-spectrum}, convexity of the scalar edges, the
bounded path-incidence norm, and
\[
 c_g\rho_q\omega=\Theta(\nu q^2).
\]
The third derivative comes only from the edge sum and is at most
\(Cc_g\rho_q\); decrease \(c_g\) to obtain the second bound in
\eqref{eq:handle-regularity}.

We next prove the minimizer statements.  Every component orthogonal to
\(h\) in \(\mathsf A\), to \(v_1\) in \(\mathsf B\), and to
\(\mathrm{span}\{v_2,\ldots,v_m\}\) in \(\mathsf W\) vanishes at the
minimizer.  Write
\[
 a=\alpha h,\qquad
 b=z_1v_1,\qquad
 u=\sum_{i=2}^m z_iv_i.
\]
After division by \(c_g\rho_q\), the scalar objective is
\[
 \sum_{i=1}^{m-1}\widehat g_{\delta_q,\omega}(z_i-z_{i+1})
 +\frac{\bar\lambda}{2}
 \left(2\alpha^2+(q+1)z_1^2-2\sqrt q\,\alpha z_1
       +\sum_{i=2}^m z_i^2\right)-\eta\alpha.\label{eq:handle-scalar}
\]
Minimizing in \(\alpha\) gives
\[
 \alpha=\frac{\sqrt q\,z_1+\eta/\bar\lambda}{2}\label{eq:handle-alpha}
\]
and, up to an additive constant, reduces
\eqref{eq:handle-scalar} to
\[
 \sum_{i=1}^{m-1}\widehat g_{\delta_q,\omega}(z_i-z_{i+1})
 +\frac S2z_1^2+\frac{\bar\lambda}{2}\sum_{i=2}^mz_i^2
 -\Gamma z_1,
 \qquad
 S:=\frac{\bar\lambda(q+2)}2.\label{eq:handle-reduced}
\]
Indeed, \(\eta\sqrt q/2=\Gamma\).

Coordinatewise projection onto \(\R_+^m\) cannot increase any edge
magnitude, any quadratic term, or the negative seed term, so the unique
minimizer has \(z_i\ge0\).  Let
\(p(t)=\widehat g'_{\delta_q,\omega}(t)\) for \(t\ge0\) and put
\(d_i=z_i-z_{i+1}\).  The first-order equations, read backward from the
last coordinate, give
\[
 p(d_i)=\bar\lambda\sum_{j=i+1}^mz_j\ge0,
 \qquad
 \Gamma=Sz_1+\bar\lambda\sum_{j=2}^mz_j.\label{eq:handle-flux}
\]
If the right-hand side is positive, odd monotonicity of the edge
derivative gives \(d_i>0\).  If it is zero, then
\(z_{i+1}=\cdots=z_m=0\) and hence \(d_i=z_i\ge0\).  Thus \(d_i\ge0\) in
all cases.  Since \(p\) is nondecreasing and
\(p(\omega)=\omega^2=\Gamma\), one also has
\[
 0\le d_i\le\omega.\label{eq:handle-edge-step}
\]
Because \(z_i\le z_1\) and \(m\le(C_m+1)q\),
\[
 z_1\ge
 \frac{\Gamma}{S+\bar\lambda(m-1)}
 \ge c\frac{\omega^2}{\bar\lambda q}
 =cC_\omega q\omega.\label{eq:handle-z1-lower}
\]
Choose \(C_\omega\) large relative to \(C_m\).  Combining
\eqref{eq:handle-z1-lower} and
\eqref{eq:handle-edge-step} gives
\[
 z_i\ge cq\omega=c\frac{D_q}{\sqrt q},
 \qquad 1\le i\le c_0q.\label{eq:handle-bulk-coordinates}
\]
The upper bound \(z_1\le\Gamma/S\le Cq\omega\), together with
\eqref{eq:handle-alpha}, gives
\[
 \sum_i z_i^2\le CD_q^2,
 \qquad |\alpha|\le CD_q.\label{eq:handle-norm}
\]
Equations \eqref{eq:handle-bulk-coordinates} and
\eqref{eq:handle-norm} prove \eqref{eq:handle-bulk}.

All non-seed terms in \eqref{eq:handle-packet} are nonnegative, hence
\[
 \mathcal H_{\mathcal V}^{\rm lin}(0)
 -\mathcal H_{\mathcal V}^{\rm lin}(y_{\mathcal V}^\star)
 \le s_q|\alpha^\star|.
\]
Finally,
\[
 \frac{s_q}{\nu D_q}
 =\frac{c_g\rho_q(2\omega^2/\sqrt q)}
 {c_g\rho_q\bar\lambda q^{3/2}\omega}
 =2\frac{\omega}{\bar\lambda q^2}
 =2C_\omega,\label{eq:seed-ratio}
\]
which proves \eqref{eq:low-seed} and \eqref{eq:handle-gap}.

For the oracle lower bound, once \(h\) is revealed the public isometry
reveals only \(v_1\).  Lemma~\ref{lem:amplifier-spectrum} shows that the
public quadratic cannot generate any tail direction.  Every unopened
edge has zero value--gradient--Hessian jet in its flat core.  The usual
resisting-frame induction therefore exposes at most one of
\(v_2,v_3,\ldots\) per response, and the bulk bound
\eqref{eq:handle-bulk-coordinates} yields distance \(\Omega(D_q)\) after
fewer than \(cq\) responses.

For randomized algorithms, localize the entire hidden edge sum outside
a public ball of radius \(R_q=CD_q\) and add the public radial collar constructed in Lemma~\ref{lem:handle-paired-shell}; leave the public amplifier global.
Before activation the linear seed is absent.  Include the query's
\(\mathsf A\) component, its pullback \(C^*b\), and its \(\mathsf W\) component
in the virtual-filtration candidate list.  Since
\(R_q/\delta_q=\Theta(q^{3/2})\), the adaptive-Haar estimate in
Lemma~\ref{lem:adaptive-haar-side} applies in dimension
\(d_{\rm amb}\ge Cq^3\log((N+2)m/\zeta)\).  All hidden edge arguments then remain in
flat cores until their predecessors are revealed.  The shell removes
hidden dependence for unbounded queries.  Appending the unqueried output
and averaging over the handle and tail frames gives one fixed instance
with the asserted failure probability.  The same argument before
activation shows that the seedless endpoint returns only the public
amplifier and public shell terms on the truncated transcript.
\end{proof}

\begin{corollary}[Private-handle parameter alignment]
\label{cor:handle-alignment}
Assume
\[
 c_q\min\{r^{2/7},\sqrt\kap\}\ge q_0,
\]
and let
\[
 q=\left\lfloor c_q\min\{r^{2/7},\sqrt\kap\}\right\rfloor,
 \qquad
 D_q=K_D\frac{\eps\sqrt\kap}{L},\label{eq:handle-alignment}
\]
where \(c_q\) is small and \(K_D\) is a sufficiently large universal
constant.  Choose the packet budget so that
\eqref{eq:handle-Dq} holds.  Then
\[
 \rho_q\asymp
 \frac{L^2q^{7/2}}{\kap^{3/2}\eps}\le\rho,
 \qquad
 \nu D_q^2\asymp\frac{\eps^2}{L},
 \qquad
 \nu q^2\lesssim L.\label{eq:handle-aligned-rates}
\]
\end{corollary}

\begin{proof}
The first identity follows from
\(D_q\asymp\nu q^{7/2}/\rho_q\) and \(\nu\asymp L/\kap\).
The definition of \(q\) gives \(q^{7/2}\lesssim r\) and
\(q^2\lesssim\kap\), proving the two inequalities.  The gap identity is
\(\nu D_q^2\asymp(L/\kap)(\eps^2\kap/L^2)\).
\end{proof}

\subsection{Bounded endpoints without a public hidden direction}
Choose a smooth odd saturated-linear function \(\varphi\) and universal
constants \(1<C_0<C_1\) such that
\[
 \varphi(t)=t\ (|t|\le C_0),\qquad
 \varphi'(t)=0\ (|t|\ge C_1),\qquad
 \sum_{j=0}^3\|\varphi^{(j)}\|_\infty\le C.\label{eq:bounded-seed-profile}
\]

Let \(s_q\) be \eqref{eq:handle-physical-seed}.  Choose a smooth odd
saturated-linear profile \(\varphi\) as in
\eqref{eq:bounded-seed-profile}, choose
\[
 A_q=C_AD_q,
 \qquad E_q=s_qA_q\asymp\nu D_q^2,\label{eq:handle-AE}
\]
with \(C_A\) sufficiently large, and let
\(\mathcal G_{\mathcal V}\) denote \eqref{eq:handle-packet} with its
linear seed removed.  Define the unlocalized endpoints
\[
 \bar H_{\mathcal V}^0(y)=\mathcal G_{\mathcal V}(y),
 \qquad
 \bar H_{\mathcal V}^1(y)=
 \mathcal G_{\mathcal V}(y)
 -E_q\varphi\!\left(\frac{\ip h a}{A_q}\right).\label{eq:handle-endpoints-unlocalized}
\]
The active minimizer lies in the linear core of \(\varphi\).  Moreover,
\(E_q/A_q^2=s_q/A_q=O(\nu/C_A)\), so for large \(C_A\) the negative seed
curvature uses only a small fraction of the amplifier's strong-convexity
margin; no rank-one correction involving \(h\) is added to the inactive
endpoint.

\begin{lemma}[Private-handle paired shell]
\label{lem:handle-paired-shell}
There are a common public cutoff and a common public radial collar that
produce localized endpoints \(H_{\mathcal V}^0,H_{\mathcal V}^1\) with
the following properties.
\begin{enumerate}[label=(\roman*)]
\item Both endpoints are globally \(\nu_0\)-strongly convex for a fixed
\(\nu_0\asymp\nu\), are \(O(\nu q^2)\)-smooth, and have
\(O(\rho_q)\)-Lipschitz Hessian.
\item Their minimizers are respectively \(0\) and the active point
\(y_{\mathcal V}^\star\) from
Theorem~\ref{thm:private-handle-packet}.
\item Their difference \(d_{\mathcal V}=H_{\mathcal V}^1-H_{\mathcal V}^0\)
obeys
\[
 \boxed{
 \sup_y\norm{D^jd_{\mathcal V}(y)}
 \le CE_qD_q^{-j},\qquad j=0,1,2,3.
 }\label{eq:handle-difference-bounds}
\]
For \(j=0\) the norm means absolute value.
\item Before activation, no selected handle or tail direction occurs in
the complete inactive transcript whenever all hidden edge arguments lie
in their flat cores.  After activation one may reveal \(h\) for free, and
each exact second-order response can reveal at most one new tail
direction.
\end{enumerate}
\end{lemma}

\begin{proof}
Keep the public amplifier \eqref{eq:public-amplifier} global.  Multiply
the hidden edge sum and the bounded seed in
\eqref{eq:handle-endpoints-unlocalized} by one public radial cutoff that
is one on \(B(0,R_q)\) and zero outside \(B(0,2R_q)\), where
\(R_q=C_RD_q\).  Choose the cutoff to be one on \(B(0,r_0R_q)\) and zero outside
\(B(0,r_1R_q)\) for fixed \(1<r_0<r_1<2\), with
\(\|D^j\chi\|\le C_jR_q^{-j}\).  Let \(\zeta\) be a smooth nondecreasing
profile, zero on \([0,1]\) and one on \([r_0,\infty)\), define
\(\eta(s)=\int_1^s u\zeta(u)\,du\) for \(s>1\) and \(\eta(s)=0\) for
\(s\le1\), and add the common public collar
\(\mathcal C_q(y)=C_{\rm col}M_qR_q^2\eta(\|y\|/R_q)\).  On the
transition annulus its radial and tangential Hessian eigenvalues are
both at least \(C_{\rm col}M_q\), while
\(\mathrm{Lip}(\nabla^2\mathcal C_q)=O(M_q/R_q)\).  Taking
\(C_{\rm col}\) larger than the universal cutoff-product constant absorbs
all negative annular curvature simultaneously for both endpoints.
On the shell, the hidden correction
has value, gradient, Hessian, and third-derivative bounds
\[
 O(M_qR_q^2),\quad O(M_qR_q),\quad O(M_q),\quad O(\rho_q),
 \qquad M_q:=\nu q^2.
\]
The seed contribution is smaller because
\[
 E_q/D_q^2=O(\nu),\qquad
 E_q/D_q^3=O(\nu/D_q)=O(\rho_q/q^{7/2}).
\]
Also \(M_q/R_q=O(\rho_q/q^{3/2})\).  Product differentiation and a collar
of curvature \(CM_q\) therefore preserve the claimed smoothness and
Hessian-Lipschitz bounds and remove all hidden dependence outside the
outer shell.  In the inner ball the seed Hessian has norm
\(O(\nu/C_A)\), so the active endpoint remains strongly convex.  On the
annulus the collar absorbs all cutoff-product negative curvature.

The active minimizer and the origin lie in the unchanged inner ball.
The derivative estimate \eqref{eq:handle-difference-bounds} follows from
\(A_q\asymp R_q\asymp D_q\) by the product rule.  Item~(iv) follows from exact flatness and the one-new-direction induction in Theorem~\ref{thm:private-handle-packet}.  Notice especially that
\(H^0\) contains no term of the form \(hh^\top\): its public quadratic
carrier is independent of the selected handle.
\end{proof}

The normalizing constants \(m_a:=\min H_{\mathcal V}^a\) are independent
of the orientation of \(\mathcal V\).  Indeed, paired orthogonal rotations
of \(\mathsf A\) and \(\mathsf B\) that intertwine \(C\), together with an
orthogonal rotation of \(\mathsf W\), map any canonical frame to any other
and leave the public amplifier invariant.

\subsection{Fenchel lift with a low-gradient endpoint difference}
Choose \(\beta\asymp L\), \(c\asymp\mu\), and
\(\lambda=\sqrt{\beta c}\), with the constant slack in
\eqref{eq:slack-budgets}.  Arrange
\(H_{\mathcal V}^a\) to have a sufficiently large constant multiple of
\(\mu\) as its strong-convexity modulus, and define
\[
 h_{\mathcal V}^a(y)=H_{\mathcal V}^a(y)-\frac c2\norm y^2,
 \qquad
 f_{\mathcal V}^a(x,y)=
 \frac\beta2\norm x^2+\lambda\ip xy-h_{\mathcal V}^a(y)+m_a.\label{eq:handle-fenchel}
\]
Let \(\Phi_{\mathcal V}^a=\max_y f_{\mathcal V}^a\) and
\(\Delta f_{\mathcal V}=f_{\mathcal V}^1-f_{\mathcal V}^0\).
Apply Lemma~\ref{lem:fenchel-distance-transfer} with \(H=H_{\mathcal V}^a\) and minimizer \(y_{\mathcal V}^\star\).  In
particular, the active value function is \(\beta\)-strongly convex, is
minimized at
\[
 x_{\mathcal V}^\star=-\frac\lambda\beta y_{\mathcal V}^\star,\label{eq:handle-xstar}
\]
and distance \(cD_q\) from \(y_{\mathcal V}^\star\) transfers to gradient
scale \(c\lambda D_q\).

\begin{lemma}[Low-gradient Fenchel packet]
\label{lem:handle-fenchel-packet}
Under the alignment \eqref{eq:handle-alignment}, after universal
rescaling the pair \eqref{eq:handle-fenchel} belongs to the original
NC--SC class and has the following properties.  Put
\[
 E:=\frac{\eps^2}{L},\qquad
 \ell:=\frac{\lambda D_q}{\beta}\asymp\frac\eps L.\label{eq:handle-E-ell}
\]
Then the active endpoint has initial gap \(O(E)\), and packet distance
\(cD_q\) transfers to value-function gradient scale \(c\lambda D_q\).
Moreover,
\[
 \boxed{
 \begin{aligned}
 \norm{\Delta f_{\mathcal V}}_\infty&\le CE,\\
 \norm{\nabla\Delta f_{\mathcal V}}_\infty
   &\le C\mu D_q\le C\frac\eps{\sqrt\kap},\\
 \norm{\nabla^2\Delta f_{\mathcal V}}_\infty
   &\le C\mu,\\
 \mathrm{Lip}(\nabla^2\Delta f_{\mathcal V})
   &\le C\frac\mu{D_q}
    \le C\frac{L^2}{\eps\kap^{3/2}}.
 \end{aligned}}\label{eq:handle-refined-difference}
\]
The difference is independent of \(x\) and depends on the private frame
only through the handle \(h\); no tail direction occurs in it.
\end{lemma}

\begin{proof}
Class membership follows from
\(\nu q^2\lesssim L\), \(\rho_q\le\rho\), and the constant-slack choice.
The gap and query lower bound are
Theorem~\ref{thm:private-handle-packet},
Lemma~\ref{lem:fenchel-distance-transfer}, and
\(\lambda D_q\asymp\eps\), together with \eqref{eq:fenchel-distance-core}.  Since the two endpoints have identical public
quadratic and identical hidden edge sum, their nonconstant difference is
only the localized bounded handle seed.  Equation
\eqref{eq:handle-difference-bounds}, together with
\(E_q\asymp\mu D_q^2\), gives all four inequalities in
\eqref{eq:handle-refined-difference}.  The scalar normalizing difference
\(m_1-m_0\) is public and has zero derivatives.
\end{proof}

\subsection{A dual relay that preserves private activation}
For \(r\ge1\), choose
\[
 a_{\rm rel}:=C_a\max\left\{1,
 \left(\frac{\kap^{3/2}}r\right)^{1/3}\right\},
 \qquad
 \sigma_{\rm rel}:=c_\sigma\frac{L}{a_{\rm rel}^2}.\label{eq:handle-relay-parameters}
\]
The gate profile and the constants in \(D_q\) are chosen so that
\(1\lesssim a_{\rm rel}\lesssim\sqrt\kap\), the relay quadratic fits in
the joint smoothness budget, \(\sigma_{\rm rel}\) has a sufficiently large constant
multiple of \(\mu\) as its \(s\)-curvature, and
\[
 \frac{L^2}{a_{\rm rel}^3\eps}\le c\rho.\label{eq:handle-relay-third}
\]
This simultaneous choice is possible under the constant-slack convention
\eqref{eq:slack-budgets}.

Let \(\vartheta\) be a smooth nonincreasing radial profile, equal to one
on a small inner ball and exactly zero outside a larger ball, with all
jets zero on both plateaus.  Define
\[
 \Theta_t(s_t)=\vartheta\!\left(
 \frac{\norm{s_t-a_{\rm rel}x_t^\star}^2}
 {a_{\rm rel}^2\ell^2}\right),
 \qquad
 \norm{D^j\Theta_t}\le
 \frac{C_j}{(a_{\rm rel}\ell)^j},\quad j=1,2,3.\label{eq:handle-relay-gate}
\]
For independent private-handle packets, set
\[
 \boxed{
 \begin{aligned}
 F_T(\mathbf x,\mathbf y,\mathbf s)
 :={}&f_1^1(x_1,y_1)
 -\frac{\sigma_{\rm rel}}2\sum_{t=1}^T
       \norm{s_t-a_{\rm rel}x_t}^2\\
 &+\sum_{t=2}^T\left[
 f_t^0(x_t,y_t)+\Theta_{t-1}(s_{t-1})
             \Delta f_t(x_t,y_t)\right]
 -K_0E\sum_{t=1}^{T-1}\Theta_t(s_t).
 \end{aligned}}\label{eq:handle-relay-family}
\]
The full dual variable is \((\mathbf y,\mathbf s)\).

\begin{lemma}[Analytic validity of the private-handle relay]
\label{lem:handle-relay-regularity}
For appropriate universal choices and bounded-degree block scaling,
\eqref{eq:handle-relay-family} is jointly \(L\)-smooth, has
\(\rho\)-Lipschitz Hessian, and is \(\mu\)-strongly concave in
\((\mathbf y,\mathbf s)\) for every \(\mathbf x\).  Moreover,
\[
 \Phi_T(\mathbf x_0)-\inf\Phi_T\le CTE\label{eq:handle-relay-gap}
\]
for the exact-inner initialization described below.
\end{lemma}

\begin{proof}
The relay quadratic has joint Hessian norm
\(\sigma_{\rm rel}(1+a_{\rm rel}^2)=O(L)\).  For one product
\(\Theta\Delta f\), the nonconstant second-derivative groups are bounded by
\[
 \frac{E}{a_{\rm rel}^2\ell^2},\qquad
 \frac{\mu D_q}{a_{\rm rel}\ell},\qquad
 \mu,\label{eq:handle-relay-second}
\]
all of which are \(O(L)\).  The four third-derivative groups are
\[
 \frac{E}{a_{\rm rel}^3\ell^3},\qquad
 \frac{\mu D_q}{a_{\rm rel}^2\ell^2},\qquad
 \frac{\mu}{a_{\rm rel}\ell},\qquad
 \frac\mu{D_q}.\label{eq:handle-relay-third-groups}
\]
After division by
\(\rho=rL^2/(\kap^{3/2}\eps)\), their ratios are, up to universal
constants,
\[
 \frac{\kap^{3/2}}{a_{\rm rel}^3r},\qquad
 \frac\kap{a_{\rm rel}^2r},\qquad
 \frac{\sqrt\kap}{a_{\rm rel}r},\qquad
 \frac1r.\label{eq:handle-relay-ratios}
\]
If \(1\le r\le\kap^{3/2}\), the first is \(O(1)\) and the next two are
\(O(r^{-1/3})\) and \(O(r^{-2/3})\); if \(r\ge\kap^{3/2}\) all are at most a
universal constant.  The reward has only the pure-gate terms.  The block
interaction graph has bounded degree, proving the global smoothness and
Hessian-Lipschitz bounds.

It remains to check dual strong concavity without any directional
rank-one reinforcement.  For one pair \((s_t,y_{t+1})\), abbreviate
\[
 p=\nabla\Theta_t(s_t),\qquad
 g=\nabla_y\Delta f_{t+1}(x_{t+1},y_{t+1}),
 \qquad
 B=\Delta f_{t+1}-K_0E.
\]
Choose \(K_0\) so that \(-C_BE\le B\le-E\).  The corresponding block of the
negative dual Hessian is
\[
 \mathcal K=
 \begin{bmatrix}
 \sigma_{\rm rel} I-B\nabla^2\Theta_t&-pg^\top\\
 -gp^\top&\nabla^2h_{t+1}^{\Theta_t}
 \end{bmatrix}.\label{eq:handle-dual-block}
\]
The profile constants and \(\sigma_{\rm rel}\) make the upper-left block at least
\(\sigma_{\rm rel} I/2\).  The endpoint construction has a fixed strong-convexity
margin, say
\(\nabla^2h_{t+1}^{\Theta_t}\succeq C_h\mu I\).  By
\eqref{eq:handle-refined-difference},
\[
 \frac{4\norm{p}^2\norm{g}^2}{\sigma_{\rm rel}}
 \le C\frac{1}{a_{\rm rel}^2\ell^2}
       \frac{\mu^2D_q^2}{L/a_{\rm rel}^2}
 \le C'\mu.\label{eq:handle-young}
\]
Choose \(C_h>C'+2\).  Young's inequality then yields
\(\mathcal K\succeq\mu I\).  The pairs \((s_t,y_{t+1})\) are disjoint, so
the whole function is \(\mu\)-strongly concave in its dual variables.

Every endpoint difference and reward is uniformly \(O(E)\).  For a fixed
primal point \(\mathbf x\), evaluate the inner maximum at
\(s_t=a_{\rm rel}x_t\); then every relay quadratic vanishes.  Each
interpolated local value satisfies
\(\Phi_t^{\alpha}(x_t)\ge \Phi_t^0(x_t)-CE\ge-CE\), after the common
normalization, while each progress reward is bounded below by
\(-K_0E\).  Therefore \(\Phi_T(\mathbf x)\ge-CTE\) for every
\(\mathbf x\).  At the initialization used below all future gates are on
the zero plateau, every future inactive endpoint is evaluated at its
normalized minimizer, and the first active packet has excess value
\(O(E)\); hence \(\Phi_T(\mathbf x_0)\le CE\).  Together these estimates
prove \eqref{eq:handle-relay-gap}.
\end{proof}

\begin{lemma}[True stationarity opens every private-handle gate]
\label{lem:handle-relay-stationarity}
With the constants chosen above, every point satisfying
\(\norm{\nabla\Phi_T(\mathbf x)}\le\eps\) has
\[
 \Theta_t(s_t^\star(\mathbf x))=1,
 \qquad 1\le t<T,\label{eq:handle-all-open}
\]
and
\[
 \norm{x_t-x_t^\star}\le b_0\ell,
 \qquad 1\le t\le T.\label{eq:handle-output-close}
\]
\end{lemma}

\begin{proof}
Let \(t\) be the first gate not on its unit plateau.  Packet \(t\) is active.
For \(t<T\), Danskin's theorem gives a scalar
\(B_t\in[-C_BE,-E]\) and the exact inner relay equation
\[
 \sigma_{\rm rel}(s_t^\star-a_{\rm rel}x_t)
 =B_t\nabla\Theta_t(s_t^\star).\label{eq:handle-relay-foc}
\]
Put \(c_t=a_{\rm rel}x_t^\star\) and
\(r_t=s_t^\star-c_t\).  Radial monotonicity gives
\(\nabla\Theta_t=-\eta_tr_t\) with
\(0\le\eta_t\le C/(a_{\rm rel}^2\ell^2)\).  Thus
\[
 a_{\rm rel}(x_t-x_t^\star)=(1-\delta_t)r_t,
 \qquad
 \delta_t=|B_t|\eta_t/\sigma_{\rm rel}\le\frac12.\label{eq:handle-collinearity}
\]
The outer block gradient is
\[
 \nabla_{x_t}\Phi_T
 =\nabla\Phi_t^1(x_t)
 +\sigma_{\rm rel} a_{\rm rel}(s_t^\star-a_{\rm rel}x_t).\label{eq:handle-outer-gradient}
\]
Both terms have nonnegative projection on \(r_t\): the second by
\eqref{eq:handle-relay-foc}, and the first by \(\beta\)-strong convexity of
\(\Phi_t^1\) and \eqref{eq:handle-collinearity}.  Outside the unit plateau,
\(\norm{r_t}\ge c_0a_{\rm rel}\ell\), so that this projection is at least
\(c\beta\ell>\eps\), after the fixed scaling in \(D_q\).  This contradicts
stationarity.  The last packet has \(s_T^\star=a_{\rm rel}x_T\) and is
handled directly by strong convexity.  On every unit plateau the gate
gradient vanishes, so \eqref{eq:handle-relay-foc} gives
\(s_t^\star=a_{\rm rel}x_t\) and the plateau radius yields
\eqref{eq:handle-output-close}.
\end{proof}

\subsection{Serial transcript and the completed randomized lower branch}
Before packet
\(t\) is activated, the complete inactive transcript contains the public
amplifier but no selected handle \(h_t\).  Once activation occurs, the
seed gradient may reveal \(h_t\) immediately; we grant it for free.  The
public map then reveals only \(v_{t,1}=C_th_t\).  By
Lemma~\ref{lem:amplifier-spectrum}, all powers of the public Hessian stay
in the two-dimensional handle--first-coordinate space.  The tail
\(v_{t,2},v_{t,3},\ldots\) can still advance only through the exactly-flat
edge chain, one direction per response.

\subsection{A fixed-instance probability space and packetwise virtual oracles}
The remaining issue is not a new geometric estimate but a quantifier
problem.  We therefore spell out the probability space and the virtual
filtrations used in the serial argument.  For this lower bound we work
with a strictly stronger augmented oracle: after a packet is activated,
each response may additionally disclose one next tail direction for
free.  Any algorithm for the original bundled oracle is a special case
that ignores this side information, so hardness for the augmented model
implies hardness for the original model.

Fix a randomized algorithm and write its private random seed as \(\xi\).
Conditioned on \(\xi\), the algorithm is deterministic: its \(n\)th query
and its final output are Borel functions of the preceding oracle
responses.  Independently of \(\xi\), for every packet \(t\in[T]\) sample
\[
 \Omega_t=(h_t,V_t),
 \qquad
 h_t\sim {\rm Haar}(\mathbb S(\mathsf A_t)),
 \qquad
 V_t=(v_{t,2},\ldots,v_{t,m})
\]
where, conditionally on \(h_t\), the tail frame is Haar-uniform in the
prescribed copy of \(\mathsf W_t\) and is independent of all other
packets.  Put \(v_{t,1}=C_th_t\).  The joint law is the product law
\[
 \mathbb P=\mathbb P_\xi\otimes\bigotimes_{t=1}^T
 \mathbb P_{\Omega_t}.\label{eq:product-probability-space}
\]
All oracle responses below include value, full gradient, and full joint
Hessian.

For a fixed packet \(t\), let
\[
 \mathcal G_t:=\sigma\!\left(\xi,(\Omega_s)_{s\ne t}\right).\label{eq:packet-side-sigma}
\]
A \(t\)-suppressed virtual oracle is defined recursively as follows.
Before packet \(t\) is activated, it returns only the public amplifier,
the public shell, and all terms from packets \(s\ne t\); the selected
handle, the bounded seed, every hidden edge of packet \(t\), and every
relay gate or reward centered at \(x_t^\star\) are replaced by their
exactly-flat zero jets.  After activation the packet-\(t\) relay gate is
still suppressed until its first virtual chamber hit.  At activation we
grant the handle \(h_t\) for free, so \(v_{t,1}=C_th_t\) becomes public.  If
the first \(k\) path directions \(v_{t,1},\ldots,v_{t,k}\) have been
disclosed, the response contains the exact jets of all edges internal to
that prefix and of the single boundary edge \((v_{t,k},v_{t,k+1})\), and
suppresses every later edge.  After the response, the protocol grants
\(v_{t,k+1}\) in full whether or not the boundary-edge jet was nonzero.
This stronger oracle discloses at most one fresh tail column per response
and, crucially, never conditions future disclosure on a ``no-hit'' event.

Let \(K_{t,n}\) be the number of disclosed path directions immediately
before round \(n\) in this virtual experiment, with \(K_{t,n}=0\) before
activation and \(K_{t,n}=1\) just after the handle is disclosed.  After
activation it increases by one after each response, capped at \(m\); this
schedule is deliberately independent of the fresh column's value.  Let
\(\widetilde R_{t,j}\) denote the virtual response at round \(j\).  The
pre-query virtual filtration is
\[
 \mathcal F^-_{t,n}:=
 \sigma\!\left(
  \mathcal G_t,
  \mathbf 1_{\{\widehat\tau_t<n\}}h_t,
  \mathbf 1_{\{K_{t,n}\ge i\}}v_{t,i}\ (2\le i\le m),
  \widetilde R_{t,1},\ldots,\widetilde R_{t,n-1}
 \right),\label{eq:packet-virtual-filtration}
\]
where \(\widehat\tau_t\) is the virtual activation time defined below.
By construction, \(\widehat\tau_t\wedge n\) is a stopping time for this
filtration, while the post-activation evolution of \(K_{t,n}\) is a
predictable function of the activation time and the round count.  Every
variable in \(\mathcal F^-_{t,n}\) is measurable with respect to
\(\mathcal G_t\) and the disclosed columns only.
Consequently, conditioned on \(\mathcal F^-_{t,n}\), the undisclosed
handle (before activation) and undisclosed tail columns retain their
Haar conditional laws in the appropriate orthogonal complements.

Run the same deterministic algorithm (after conditioning on \(\xi\))
against the \(t\)-suppressed virtual oracle.  At round \(n\) form the finite
candidate list \(\mathcal C_{t,n}\) from the packet-\(t\) components of
this virtual query.  If
\(w_{t,n}:=-\beta x_{t,n}/\lambda\) and
\(\bar r_{t,n}:=-(\beta/\lambda)s_{t,n}/a_{\rm rel}\), include, in the
corresponding hidden spaces,
\[
 \begin{gathered}
 a_{t,n},\ C_t^*b_{t,n},\ u_{t,n},\\
 \Pi_{\mathsf A_t}w_{t,n},\ C_t^*\Pi_{\mathsf B_t}w_{t,n},\
 \Pi_{\mathsf W_t}w_{t,n},\\
 \Pi_{\mathsf A_t}\bar r_{t,n},\ C_t^*\Pi_{\mathsf B_t}\bar r_{t,n},\
 \Pi_{\mathsf W_t}\bar r_{t,n}.
 \end{gathered}\label{eq:packet-candidate-list}
\]
The boundedness reduction is explicit.  Fix a public
\(R_{\rm clip}=C_{\rm clip}D_q\) with
\(C_{\rm clip}\ge 2C_R\), and large enough to exceed four times every
canonical packet-center norm and every packet chamber radius.  For a raw candidate
\(z\), retain \(z\) when \(\norm z\le R_{\rm clip}\) and otherwise replace it
by zero in the concentration list.  This replacement loses no possible
first hit: a raw packet-coordinate candidate with norm above
\(R_{\rm clip}\) is deterministically outside the \(O(D_q)\) chamber around
the hidden minimizer; a query whose packet dual block lies outside the
outer shell has no hidden packet jet at all; and the compactly supported
relay gate is zero for a mapped relay candidate outside the same public
radius.  Thus every candidate to which concentration is applied is
measurable and has norm at most \(R_{\rm clip}=O(D_q)\), while every
discarded candidate is already certified harmless without using its
orientation.

If packet \(t\) has not been activated by the terminal round, the similarly
decomposed and clipped final output is measurable with respect to its
suppressed transcript and is adjoined as a terminal candidate.  After
activation the final output is handled by the same adaptive list at
round \(N+1\).

\begin{lemma}[Packetwise real--virtual coincidence]
\label{lem:packet-real-virtual}
Fix \(t\) and condition on \(\mathcal G_t\).  There is an event
\(\mathcal E_t\) satisfying
\[
 \mathbb P(\mathcal E_t^c\mid\mathcal G_t)\le\zeta
 \quad\text{a.s.}\label{eq:packet-good-probability}
\]
provided
\[
 d_t-m\ge Cq^3
 \log\!\left(\frac{C(N+2)q}{\zeta}\right),\label{eq:packet-dimension-v12}
\]
such that the following statements hold simultaneously.
\begin{enumerate}[label=(\roman*)]
\item Before activation, the real and \(t\)-suppressed virtual responses
coincide at every round, and no query lies in the packet chamber
\(B(x_t^\star,b_0\ell)\) or the relay chamber
\(B(a_{\rm rel}x_t^\star,b_1a_{\rm rel}\ell)\).  If activation never
occurs, the same statement holds for the terminal candidate.
\item At activation the only newly frame-dependent object in the full
response is the handle \(h_t\) (equivalently \(v_{t,1}\)); no tail column is
revealed.
\item After activation, if \(K_{t,n}=k\), the real response is simulated by
the stronger virtual response that includes the one boundary edge and
then grants \(v_{t,k+1}\) for free.  Thus one response advances the
disclosed prefix by at most one.  Before \(cq\) post-activation responses,
no query, relay candidate, or terminal candidate enters the
corresponding chamber.
\item The conclusions remain valid for arbitrary activation weights in
\([0,1]\), all product-rule cross-Hessian blocks, and all public-shell
transition terms.
\end{enumerate}
\end{lemma}

\begin{proof}
We give the induction because it is the point at which a posterior-Haar
argument can otherwise become circular.  Before activation, apply
Lemma~\ref{lem:adaptive-haar-side} to the candidates
\eqref{eq:packet-candidate-list}, first with the undisclosed handle and
then with all undisclosed tail columns.  Here \(R=CD_q\) and the
exactly-flat width is \(s=cD_q/q^{3/2}\), so
\(R^2/s^2=O(q^3)\) and \eqref{eq:packet-dimension-v12} makes the total
conditional failure probability at most \(\zeta\).  On the resulting
event, every hidden handle argument and every hidden edge argument lies
strictly inside its flat core.  Hence their value, gradient, and Hessian
are all exactly zero.  Outside the truncation ball the common public
shell removes every hidden term.  Thus the real and virtual responses
are identical, proving the induction base and the first assertion.

The bounded active--inactive endpoint difference depends on the private
frame only through the handle.  Therefore at activation it may reveal
\(h_t\), which we grant for free, but it cannot reveal any tail column.
Conditioned on \(\mathcal F^-_{t,n}\) after this disclosure, the remaining
tail frame is still Haar in the orthogonal complement of the disclosed
columns.

Assume inductively that \(v_{t,1},\ldots,v_{t,k}\) have been disclosed and
that the real transcript is simulated by the stronger virtual
transcript before round \(n\).  The algorithm's query is then
\(\mathcal F^-_{t,n}\)-measurable and is selected before
\(v_{t,k+1}\) is disclosed.  Apply the conditional spherical bound to all
undisclosed columns and all candidates in
\eqref{eq:packet-candidate-list}.  Every edge strictly beyond the
boundary edge remains in its exactly-flat core.  The virtual response
includes the exact boundary-edge jet, whose range is contained in the
span of the disclosed prefix and the single fresh column
\(v_{t,k+1}\); it then grants that column in full even when this jet is
zero.  Hence the response depends on at most one fresh tail column and
cannot skip a column, while the next pre-query history contains the
fresh column explicitly rather than negative information about it.  The
public amplifier has the two-dimensional Krylov closure of
Lemma~\ref{lem:amplifier-spectrum} and contains no tail direction.

Partial activation and relay differentiation introduce only products of
jets of the endpoint difference with jets of the gate.  The endpoint
difference contains the handle but no tail direction, while every gate
is a public radial function of a candidate in
\eqref{eq:packet-candidate-list}.  Hence these product-rule tensors are
measurable with respect to the disclosed filtration and cannot disclose
an additional tail column.  The same statement holds on the shell
annulus because the cutoff and collar are public radial functions.
This proves response-by-response simulation by the stronger oracle and
the one-column increment.  Since the virtual oracle only gives the
algorithm additional information, a lower bound for it is also a lower
bound for the original bundled oracle.

Finally, the active minimizer has an \(\Omega(D_q)\) component on a
constant fraction of the first \(q\) tail directions by
\eqref{eq:handle-bulk}.  Before \(cq\) responses, a constant fraction of
those columns remains undisclosed.  The concentration event gives
\(o(D_q/\sqrt q)\) projection of every candidate on each undisclosed
column.  Therefore
\[
 \|w-y_t^\star\|^2
 \ge\sum_{i\in I_{\rm unrev}}
 \left(|\langle v_{t,i},y_t^\star\rangle|
       -|\langle v_{t,i},w\rangle|\right)^2
 \ge cD_q^2.\label{eq:packet-bulk-separation-v12}
\]
The public Fenchel scaling \(x\mapsto-\beta x/\lambda\) converts
this into distance \(c\ell\) from \(x_t^\star\).  The same scaling applied
to \(s_t/a_{\rm rel}\) converts the relay chamber exactly into the packet
coordinate test used above.  Taking \(b_0,b_1\) below this universal separation proves
all chamber claims, including the appended unqueried output.
\end{proof}

\begin{lemma}[Uniform conditional first-hit increment]
\label{lem:private-handle-first-hit}
For packet \(t\), let \(\tau_t\) be its query activation time, with
\(\tau_1=0\) and, for \(t\ge2\), \(\tau_t\) the first query round at which
\(s_{t-1}/a_{\rm rel}\) enters the relay chamber of packet \(t-1\); put
\(\tau_t=\infty\) if no such query exists.  Let \(\widehat\tau_t\) be the
corresponding time in the packet-suppressed virtual transcript.
Conditionally on \(\mathcal G_t\), and uniformly over every value of the
random time \(\widehat\tau_t\), define \(\mathcal B_t\) to be the union
of the following three events: real and virtual activation disagree
before the first hit; \(\tau_{t+1}<\tau_t+cq-C_0\); or the terminal
output hits packet \(t\) before \(\tau_t+cq-C_0\).  Then
\[
 \mathbb P(\mathcal B_t\mid\mathcal G_t)\le\zeta.\label{eq:conditional-increment-bound}
\]
For \(t=T\), interpret \(\tau_{T+1}=N+1\) when the terminal output lies in
its packet chamber and \(\tau_{T+1}=\infty\) otherwise.
\end{lemma}

\begin{proof}
The truncated virtual activation time
\(\widehat\tau_t\wedge(N+1)\) is a bounded stopping time for
\eqref{eq:packet-virtual-filtration} and is measurable without using an
undisclosed variable of packet \(t\).  Use a regular conditional
distribution given the side sigma-field and then condition further on
the stopped virtual history
\(\mathcal F^-_{t,\widehat\tau_t}\).  The estimate in
Lemma~\ref{lem:packet-real-virtual} is uniform in every realization of
that history; the optional-sampling identity for bounded stopping times
and the tower property therefore preserve the same failure bound after
averaging over the random activation time.  On the good event, the original transcript is simulated by the stronger
virtual transcript before the first hit, and the activation rule is the
same in both experiments, so \(\tau_t=\widehat\tau_t\).  The post-activation part of the same lemma
then prohibits both the next relay hit and a terminal packet hit for
\(cq-C_0\) further responses.  No independence assumption is used.
\end{proof}

\begin{theorem}[Fixed-instance serial product theorem]
\label{thm:private-handle-product}
Let \(T\ge1\) and \(q\ge q_0\), and let the \(T\) private-handle packets be
combined by \eqref{eq:handle-relay-family}.  Fix a target budget
\(N_0\) and choose every ambient block dimension so that
\eqref{eq:packet-dimension-v12} holds with \(N=N_0\) and
\(\zeta=(3T)^{-1}\).  For every randomized algorithm making
\(N\le N_0\) bundled queries, there is a deterministic choice of all
packet frames,
independent of the realized algorithmic seed, such that
\[
 N<cTq
 \quad\Longrightarrow\quad
 \mathbb P_\xi\{\|\nabla\Phi_T(\widehat{\mathbf x})\|>\eps\}
 \ge\frac23 .\label{eq:fixed-instance-product-lower}
\]
\end{theorem}

\begin{proof}
Work first under the product law
\eqref{eq:product-probability-space}.  For packet \(t\), apply
Lemma~\ref{lem:private-handle-first-hit} with conditional failure
probability \((3T)^{-1}\).  Since the estimate holds almost surely given
\(\mathcal G_t\), the tower property gives the same unconditional bound.
A union bound, without any independence assumption between the good
events, yields with probability at least \(2/3\)
\[
 \tau_{t+1}\ge\tau_t+cq-C_0,
 \qquad 1\le t\le T.\label{eq:serial-increments-v12}
\]

If the algorithm's terminal output were \(\eps\)-stationary,
Lemma~\ref{lem:handle-relay-stationarity} would place every output block
inside its packet chamber.  On the good event, a terminal hit of an
unactivated packet is impossible by
Lemma~\ref{lem:packet-real-virtual}(i); hence every packet must have
been activated by a genuine query.  For the last packet the appended
output is the hit at time \(N+1\).  Iterating
\eqref{eq:serial-increments-v12} therefore gives
\[
 N+1\ge T(cq-C_0).\label{eq:terminal-serial-count}
\]
For \(q\ge q_0\) the right side is at least \(c'Tq\).  Thus, if
\(N<c'Tq\), the joint probability over frames and the algorithmic seed of
success is at most \(1/3\).

Let \(p(\Omega_1,\ldots,\Omega_T)\) be the success probability over
\(\xi\) for the deterministic instance specified by the frame tuple.
The preceding bound states
\[
 \mathbb E_{\Omega_1,\ldots,\Omega_T}p(\Omega_1,\ldots,\Omega_T)
 \le\frac13.
\]
Hence some deterministic frame tuple satisfies \(p\le1/3\).  This tuple
is selected before, and independently of, the realized algorithmic
seed.  It is the required fixed hard instance.
\end{proof}

\begin{theorem}[Private-handle dual-relay lower bound]
\label{thm:private-handle-composition}
Assume the nontrivial regime \(A_1=\Delta L/\eps^2\ge C\) and \(r\ge1\).
In unrestricted dimension, every randomized bundled full-Hessian
algorithm that succeeds with probability at least \(2/3\) on every fixed
instance requires
\[
 \boxed{
 \Mrand(\eps)
 \ge cA_1\min\{r^{2/7},\sqrt\kap\}.
 }\label{eq:private-handle-lower}
\]
The hard family admits exact initialization \(R_y=0\), and the fixed
instance is independent of the algorithm's realized random seed.
\end{theorem}

\begin{proof}
Put \(m_q:=\min\{r^{2/7},\sqrt\kap\}\).  If
\(c_qm_q<q_0\), then \(m_q<q_0/c_q\) is bounded by a universal
constant, so the independent all-derivative baseline
\eqref{eq:all-derivative-baseline} proves the displayed bound (with exact
dual initialization) after adjusting its universal constant.  Hence assume
\(c_qm_q\ge q_0\), and take
\[
 T=\left\lfloor c_TA_1\right\rfloor,
 \qquad
 q=\left\lfloor c_qm_q\right\rfloor\ge q_0.\label{eq:handle-Tq}
\]
Set \(N_0=\lfloor c_0Tq\rfloor\) with a sufficiently small
universal \(c_0\), and choose the block dimensions required by
Theorem~\ref{thm:private-handle-product} for this target budget.  Lemmas
\ref{lem:handle-relay-regularity} and
\ref{lem:handle-relay-stationarity} give class membership, the gap
budget, and the implication from stationarity to all packet chambers.
The product theorem yields \(\Omega(Tq)\) queries, which is the displayed
rate.

For exact initialization set
\[
 x_{1,0}=\lambda^{-1}\nabla h_1^1(0),
 \qquad x_{t,0}=0\ (t\ge2),
 \qquad y_{t,0}=0,\label{eq:handle-exact-init-x}
\]
\[
 s_{1,0}=a_{\rm rel}x_{1,0},
 \qquad s_{t,0}=0\ (t\ge2).\label{eq:handle-exact-init-s}
\]
The first relation satisfies the active packet inner equation at the
origin, all inactive packet gradients vanish there, and the gate-support
constant is chosen below the fixed unresolved-tail separation.  Thus all
relay gates and their derivatives vanish at the displayed point.
Global dual strong concavity makes \((\mathbf y_0,\mathbf s_0)\) the exact
inner maximizer, so \(R_y=0\).  The first active value function is
\(O(L)\)-smooth and the initial primal displacement is \(O(\ell)\), giving
additional gap \(O(L\ell^2)=O(E)\).  Together with
\eqref{eq:handle-relay-gap} and \(T\asymp A_1\), this respects the initial
gap \(\Delta\).
\end{proof}

\begin{corollary}[Randomized lower frontier]
\label{cor:complete-lower}
In the nontrivial regime \(A_1\ge C\), the exact bundled full-Hessian model
satisfies
\[
 \boxed{
 \Mrand(\eps)\ge
 cA_1\Psi_\kap(r)
 }\label{eq:complete-lower}
\]
for every \(r>0\), with exact initialization \(R_y=0\).  Here
\(\Psi_\kap\) is \eqref{eq:Psi}.
\end{corollary}

\begin{proof}
For \(r\le1\), use the all-derivative lift
\eqref{eq:all-derivative-baseline}.  For \(r\ge1\), use
Theorem~\ref{thm:private-handle-composition}.  The two expressions agree
at \(r=1\), and the packet multiplier saturates at \(\sqrt\kap\) when
\(r=\kap^{7/4}\).
\end{proof}

\bibliographystyle{unsrtnat}
\bibliography{references}
\end{document}